%% file: Delay_Parab_arx.tex
\documentclass[reqno]{amsart}

\usepackage{amscd,amsthm,amsmath,amssymb,mathtools}
\usepackage{upgreek,bm}
\usepackage{verbatim}
\usepackage{color}
\usepackage{url}
\usepackage{enumerate,enumitem}
\usepackage{graphicx}

\usepackage{algorithm}
\usepackage{algorithmic}
\usepackage{epstopdf,epsfig,subfigure}
\usepackage{curve2e}
\usepackage{mathrsfs}
\usepackage{array}
\usepackage{stackrel}

\usepackage[numbers,sort&compress]{natbib}

\usepackage{setspace}
\usepackage[
    bookmarks=true,         
    unicode=false,          
    pdftoolbar=true,        
    pdfmenubar=true,        
    pdffitwindow=false,     
    pdfstartview={FitH},    
    pdftitle={Stabilization of nonautonomous parabolic equations with delayed input},    
    pdfauthor={Author},     
    pdfsubject={Subject},   
    pdfcreator={Creator},   
    pdfproducer={Producer}, 
    pdfkeywords={stabilization of parabolic equations, feedback-input, Luenberger observer, delayed input}, 
    pdfnewwindow=true,      
    colorlinks=true,       
    linkcolor=red,          
    citecolor=blue,        
    filecolor=magenta,      
    urlcolor=cyan,           
hypertexnames=false
]{hyperref}

\theoremstyle{plain}

\newtheorem{theorem}{Theorem}[section]

\newtheorem{lemma}[theorem]{Lemma}

\newtheorem{assumption}[theorem]{Assumption}

\theoremstyle{definition}
\newtheorem{definition}[theorem]{Definition}
\newtheorem{remark}[theorem]{Remark}

\numberwithin{equation}{section}

\usepackage{geometry}

\input{Mathcommands}

\pdfoutput=1

\begin{document}
\title{Stabilization of nonautonomous linear parabolic equations with inputs subject to  time-delay}
\author{Karl Kunisch$^{\tt1,2}$}
\author{S\'ergio S.~Rodrigues$^{\tt1}$}
\thanks{
\vspace{-1em}\newline\noindent
\newline\noindent
$^{\tt1}$ Johann Radon Inst. Comput. Appl. Math.,
  \"OAW, Altenbergerstr. 69, 4040 Linz, Austria.
 \newline\noindent  
  $^{\tt2}$ Inst. Math. and  Scient. Comput., Karl-Franzens University Graz, Heinrichstr. 36, 8010 Graz, Austria.
\quad
\newline\noindent
{\sc Emails}:
{\small\tt karl.kunisch@uni-graz.at,\quad sergio.rodrigues@ricam.oeaw.ac.at}%
}

\begin{abstract}
The stabilization of nonautonomous parabolic equations is achieved by feedback inputs tuning a finite number of actuators, where it is assumed that the input is subject to a time delay. To overcome destabilizing effects of the time delay, the input is based on a prediction of the state at a future time. This prediction is computed depending on a state-estimate at the current time, which in turn is provided by a Luenberger observer. 
The observer is designed using the output of measurements performed by a finite number of sensors. The asymptotic behavior of the resulting coupled system is investigated.
Numerical simulations are presented validating the theoretical findings, including tests showing the response against sensor measurement errors.
\end{abstract}

\maketitle

\pagestyle{myheadings} \thispagestyle{plain} \markboth{\sc  K. Kunisch and S. S.
Rodrigues}{\sc }

\section{Introduction}
The stabilization of unstable processes is an important problem in many applications. It involves the design of stabilizing control inputs, often demanded in feedback form $u(t)=K(t,y(t))$ due to their robustness properties against small errors which are ubiquitous in practice.

A system, which is stable for a given feedback input, may become unstable if the input is  subject to a time-delay~$\tau>0$, resulting in $u^{[\tau]}(t)=K(t-\tau,y(t-\tau))$. In this case, such time-delays cannot be neglected. We need  to construct~$K=K_\tau$ depending on~$\tau$, leading to $u_\tau^{[\tau]}(t)=K_\tau(t-\tau,y(t-\tau))$. 
If the state~$y(t)$ is not available, we take the   input~$\widehat u_\tau^{[\tau]}(t)=K_\tau(t-\tau,\widehat y(t-\tau))$, using a state estimate~$\widehat y(t-\tau)$.
Throughout, we denote $g^{[\tau]}(t)\coloneqq g(t-\tau)$. The superscript~$[\tau]$, in~$\widehat u_\tau^{[\tau]}$,  means that the input~$u_\tau(t)$ is delayed by time~$\tau$; the subscript means that the feedback-input operator~$K_\tau$ will be constructed depending on~$\tau$, based on a predictor.

\subsection{On the novelties} 
An important aspect of our work is presented by the fact that delayed stabilizing dynamic output-feedbacks are investigated for nonautonomous infinite-dimensional systems. Furthermore, we present a strategy able to tackle general higher-dimensional spatial domains.

For the particular case of autonomous (time-invariant) dynamics and if the state~$y(t)$ is available, we can mention~\cite{ManitiusOlbrot79,KwonPearson80} for finite-dimensional systems, and~\cite{DjebourTakahValein22,LhachemiPrieur21,LhachemiPrieur22,LhachemiPrieurShorten19,LhachemiPrieurTrelat21,Munteanu23} for infinite-dimensional systems. For nonautonomous dynamics we refer to the seminal work~\cite{Artstein82} for finite-dimensional systems. In the context of infinite-dimensional systems, the analysis of nonautonomous dynamics with delayed inputs has been addressed in~\cite{Haad06}, but we are not aware of works on stabilizability results with delayed feedback inputs, in this context. Below, we shall see that some tools used for autonomous dynamics and finite-dimensional systems are not appropriate to deal with nonautonomous infinite-dimensional systems.

In the particular case of one-dimensional spatial domains and boundary controls, we have at our disposal a powerful technique using appropriate backstepping transformations; see the survey~\cite{VazquezAuriolBArgomedoKrstic26}. Further, the delay can be tackled by the use of an auxiliary variable and considering an extended system. We refer the reader to~\cite[Sect.~2]{Krstic09}.  These techniques have been further explored in~\cite{QiKrstic21,WangDiagneQi22,WangQiDiagne21,WangDiagneKrstic25}, for example. In higher-dimensional domains the use of backstepping techniques is still limited to particular geometries (cf.~\cite[Sect.~7.5]{VazquezAuriolBArgomedoKrstic26}); see~\cite{JadachowskiMeurerKugi15} for rectangles/boxes and~\cite{VasquezKrstic16} for disks/balls; see also~\cite[Sect.~5]{Vasquez25} where it is mentioned that a domain extension technique could be explored since it transforms boundary control problems in a given spatial domain into boundary/internal control problems in larger domains (cf.~\cite[Sect.~3]{FursikovImanuvilov99}, \cite{Rod14-na}). In this manuscript we consider internal controls and present a design strategy able to tackle  general polygonal/polyhedral domains.

A predictor will be utilized to counteract the destabilizing effects of the delay  (cf.~\cite{Krstic08}, \cite{Artstein82}). At the same time, to account for the possibility that   the state~$y(t)$ is not available, the proposed  strategy will involve a Luenberger observer, providing us with an estimate~$\widehat y(t)$ for~$y(t)$. This results in a system coupling the state dynamics, the predictor, and the observer, which we believe has not been addressed in the literature, in the context of nonautonomous systems, not even in the case of finite-dimensional systems. For the autonomous case we can mention~\cite{LhachemiPrieur22} where the eigenpairs of the time-independent diffusion-reaction-convection operator have been explored. Further, we shall give the nominal feedback-input and output-injection operators in an explicit form simple to compute numerically, which is important in practice.

\subsection{Destabilizing effect of delays}\label{sS:intro-delayDest}  

Stabilizing feedback inputs can become destabilizing under  arbitrarily small time-delays in the context of partial differential equations ({\sc pde}s) of hyperbolic type (cf.~\cite[Thm.~3.1]{Datko88}). Such a destabilizing effect is also present for {\sc pde}s of parabolic type, if the time-delay~$\tau$ is larger than a suitable threshold~$\tau_*>0$ depending on the free dynamics. Further,  given~$\tau_*$ we can find a parabolic model with feedback input delay threshold~$\tau_*$. Since the parabolic case may not have been explicitly addressed in the literature, we shall give an example later on.

\subsection{Prediction and time irreversibility}\label{sS:intro-prediction}
The input~$u(t)$ at time~$t$  will be based on a prediction~$y_\fkp(t+\tau)$ of the unknown state~$y(t+\tau)$, at the future time~$t+\tau$,  using a state estimate~$\widehat y(t)$ available at time~$t$.
In the autonomous case, with free dynamics $\dot y=\clA y$,  and if time can be reversed and if the state is known, $\widehat y(t)=y(t)$, the strategy is comparable to the ones followed in~\cite{ManitiusOlbrot79,Artstein82,KwonPearson80,Krstic08,BPetriPrieurTrelat18} for finite-dimensional systems, where the state~$y(t)$ is transformed into
\begin{equation}\label{intro-pre-psi}
 \psi(t)\coloneqq y(t)+{\textstyle\int_{t-\tau}^t} \rme^{(t-\tau-s)\clA}Bu(s)\rmd s,
 \end{equation}
 and the feedback input is sought in the form~$u(t)= \widetilde K(\psi(t))$. 
 Indeed, note that from~\eqref{intro-pre-psi}
 we find
 \begin{align}
 \rme^{\tau\clA}\psi(t)&=\rme^{\tau  \clA}y(t)+{\textstyle\int_{t-\tau}^t} \rme^{(t-s)\clA}Bu(s)\rmd s\notag\\
 &=\rme^{\tau \clA}y(t)+{\textstyle\int_{t}^{t+\tau}} \rme^{(t+\tau-r)\clA}Bu(r-\tau)\rmd r,\notag
 \end{align}
which shows that~$\rme^{\tau\clA}\psi(t)=y(t+\tau)$. That is, the input~$u(t)=\widetilde K(\psi(t))=K_\tau(y(t+\tau))$ is a function of the state at the future time~$t+\tau$, with~$K_\tau\coloneqq \widetilde K\circ\rme^{\tau \clA}$.

In case the state~$y(t)$ is not available we will use a state estimate~$\widehat y(t)$ instead, by setting 
\begin{align}
\hspace{-.3em}y_\fkp (t\!+\!\tau)&\coloneqq \rme^{\tau\clA}\widehat y(t)\!+\!{\textstyle\int_{t}^{t+\tau}}  \rme^{(t+\tau-r)\clA}Bu(r\!-\!\tau)\rmd r,\label{intro-pre-psi-exp}\hspace{-.5em}
 \end{align}
resulting in the input~$u(t)=K_\tau(y_\fkp(t+\tau))$. We use~\eqref{intro-pre-psi-exp} instead of~\eqref{intro-pre-psi} because, in general, for parabolic-like equations with state~$y(t)$ evolving in an infinite-dimensional space, the reversibility of time does not hold and thus  the integrand in~\eqref{intro-pre-psi} is not necessarily well defined. Note that $t-\tau-s<0$ for~$s>t-\tau$.

\subsection{Stabilizability and nonautonomous dynamics}\label{sS:intro-stabilizability}
In the particular case of  parabolic  autonomous dynamics, the spectral properties of the time-independent operator~$\clA$ can be used to reduce the problem of stabilizability 
to the stabilizability of a finite-dimensional system.
This strategy is followed, for example, in~\cite[Sect.~2]{DjebourTakahValein22}, \cite[Sect.~3]{LhachemiPrieurShorten19}, \cite[Sect.~3]{Munteanu23}.
Spectral properties of time-dependent operators~$\clA(t)$,  however,  are not an appropriate tool to investigate stability properties, see~\cite{Wu74}.  

Here, we consider  linear parabolic equations
 \begin{subequations}\label{sys-y-intro0} 
 \begin{align}
 &\dot y +Ay+ A_{\rm rc}y =B\widehat u_\tau^{[\tau]},\quad y(0)= y_0,\label{sys-y-intro0-dyn} 
 \intertext{where  the state~$y(t)$, $t\ge0$, evolves in a real separable Hilbert space~$H$. Here,~$A$ is a symmetric diffusion-like operator, $A_{\rm rc}=A_{\rm rc}(t)$ is a reaction--convection-like linear operator, $B$ is a linear control operator, and we seek a  feedback input~$\widehat u_\tau(t)=(\widehat u_{\tau,1}(t),\dots,\widehat u_{\tau,m}(t))\in\bbR^m$ which is subject to a time-$\tau$ delay,}
 &\widehat u_\tau^{[\tau]}(t)\coloneqq u_\tau(t-\tau),\\
 &\mbox{with}\quad \widehat u_\tau(t)\coloneqq
 \begin{cases}
 0,&\mbox{for } t\in[-\tau,0),\\
 K_\tau(t,\widehat y(t)),&\mbox{for } t\in[0,\infty),
 \end{cases}
 \end{align} 
 where~$\widehat y(t)$ is an estimate for~$y(t)$. The goal is to design a stabilizing operator~$K_\tau$ so that~$y(t)\to0$ converges to zero as~$t\to\infty$.
The control forcing $B\widehat  u_\tau^{[\tau]}(t)$ is, at each time~$t\ge0$, a linear combination 
 \begin{equation}
 B \widehat u_\tau^{[\tau]}(t)\coloneqq{\textstyle \sum\limits_{i=1}^m}\widehat u^{[\tau]}_{\tau,i}(t)\Phi_i\label{B-intro}
\end{equation}
of a finite number~$m$ of given actuators~$\Phi_i\in H$.
\end{subequations}
Given a feedback-input operator~$K$ so that the nominal system
 \begin{equation}\label{sys-intro-nominal}
 \dot y +Ay+A_{\rm rc}y =BK(t,y(t)),\quad y(0)= y_0,
 \end{equation}
 is stable, with state feedback input~$u(t)=K(t,y(t))$, the  actual dynamics is of the form
  \begin{equation}\notag
 \dot y +Ay+A_{\rm rc}y =BK(t-\tau,\widehat y(t-\tau)),\qquad y(0)= y_0,
 \end{equation}
with delayed input~$\widehat u^{[\tau]}(t)=K(t-\tau,\widehat y(t-\tau))$.
By nominal system/dynamics we mean the system/dynamics under no disturbances, namely, the ideal (nonrealistic) situation in~\eqref{sys-intro-nominal}, with vanishing input delay $\tau=0$ and with availability of the full state~$y(t)$.

\subsection{Detectability}
The estimate~$\widehat y$ is given by a Luenberger observer 
 \begin{align}\label{observer-intro}
 \hspace{-.5em}&\dot{\widehat y}+A\widehat y+ A_{\rm rc}\widehat y =Bu_\tau^{[\tau]}\!+\!LW(\widehat y-y),\hspace{.7em}\widehat y(0)\!=\! \widehat y_0,\hspace{-1em}
 \end{align}
 with the input as in~\eqref{sys-y-intro0}. Here~$W\colon H\to\bbR^s$ is a linear output operator, representing the measurements $w(t)=Wy(t)=(w_1(t),\dots,w_s(t))$ obtained by a finite number~$s$ of sensors, at time~$t$, and~$L\colon \bbR^s\to H$ is an output-injection operator to be designed. The initial state~$\widehat y_0$ of the observer is at our disposal. It can be chosen as an initial guess we might have for~$y_0$. The goal is to design a detecting operator~$L$ so that~$\widehat y(t)\to y(t)$ as~$t\to\infty$.

\subsection{Contents}
In Section~\ref{sS:Assum} we present the general assumptions on the tuple of operators~$(A,A_{\rm rc},B,W)$, in particular, to guarantee  the well posedness of the involved problems and the possibility of the designs of the sought feedback-input operator~$K$ and of the sought   output-injection operator~$L$.
Section~\ref{sS:Aux-nominal} gathers auxiliary results concerning the regularity and boundedness of the solutions.
The destabilizing effects of delayed feedback inputs is shown in Section~\ref{S:dest-delayK} through an example of a scalar {\sc ode}, satisfying the assumptions made in Section~\ref{sS:Assum}.
This example is used, in the Appendix, to show the analogous destabilizing effects for a more general class of parabolic equations, including reaction-diffusion equations.
The proof of the main result, stated in Theorem~\ref{T:main}, is given in Section~\ref{S:stabil}, for the considered class of abstract evolution equations. Section~\ref{S:observers} focuses on the case of exponential Luenberger observers.
In Section~\ref{S:satAssum} it is shown that the abstract assumptions are satisfied by scalar parabolic equations.
Results of numerical simulations are discussed in Section~\ref{S:numer}.
Comments on the results are given in Section~\ref{S:finremks}.

\section{Assumptions and auxiliary results}\label{S:AssumAux}
Hereafter, $H$ and~$V$ are two real separable Hilbert spaces. The former is considered as pivot space, $H=H'$.  The identity operator shall be denoted by~$\Id$.

\begin{definition}\label{D:clAstab}
Let~$D\ge1$ and~$\mu\ge0$. The operator $\clA\coloneqq\clA(t)\in\clL(V,V')$ is called~$(D,\mu)$-stable if, for all $z_0\in H$,  the solution of the system~$\bigl\{\dot z=\clA z,\; z(0)=z_0\bigr\}$
satisfies $\norm{z(t)}{H}\le D\rme^{-\mu(t-s)}\norm{z(s)}{H}$ for all $t\ge s\ge 0$.
\end{definition}

\subsection{Assumptions}\label{sS:Assum}
We start by making general assumptions on the operators~$A$ and~$A_{\rm rc}$, defining the free dynamics.

\begin{assumption}\label{A:HV}
The inclusion $V\subseteq H$ is dense, continuous, and compact.
\end{assumption}

\begin{assumption}\label{A:A}
The operator~$A\in\clL(V,V')$ is symmetric and~$(y,z)\mapsto\langle A y,z\rangle_{V',V}$ is a complete scalar product on~$V.$
\end{assumption}

We suppose that~$V$ is endowed with the scalar product~$(y,z)_V\coloneqq\langle Ay,z\rangle_{V',V}$,
which again makes~$V$ a Hilbert space.
Then, necessarily $A\colon V\to V'$ is an isometry.
The domain of~$A$ in~$H$ is denoted by
$\rmD(A)\coloneqq\{z\in H\mid Az\in H\}$
and it is  endowed with the scalar product
$(z,w)_{\rmD(A)}\coloneqq(Az,Aw)_H,$
which defines a norm equivalent to the graph norm.
\begin{assumption}\label{A:Arc}
We have that~$A_{\rm rc}\in\clC([0,\infty),\clL(V,H)+\clL(H,V'))$ and that it is bounded, $\sup\limits_{t\ge0}\norm{A_{\rm rc}(t)}{\clL(V,H)+\clL(H,V')}\eqqcolon C_{\rm rc}<\infty$.
\end{assumption}

The following assumptions involve the given control and output operators, $B$ and~$W$.
\begin{assumption}\label{A:Kstab}
There is a feedback-input operator $K\in\clL(H,\bbR^m)$ such that $\clA_{BK}\coloneqq-A-A_{\rm rc}+BK\in\clL(V,V')$ is~$(D_1,\mu_1)$-stable, for some~$D_1\ge1$ and~$\mu_1>0$.
\end{assumption}

\begin{assumption}\label{A:detect}
There is an output-injection  operator~$L\in\clL(\bbR^s,H)$ such that $\clA_{LW}\coloneqq-A- A_{\rm rc}(t)+LW$ is $(C,0)$-stable, for some~$C\ge1$.
\end{assumption}

\subsection{Auxiliary results}\label{sS:Aux-nominal}
We consider the dynamics under a general forcing~$f$,
\begin{equation}\label{sys-yf} 
 \dot y +Ay+ A_{\rm rc}y =f,\qquad y(0)= y_0\in H,
 \end{equation}
for time~$t\ge0$. Let us denote~$I_s^T\coloneqq(s,s+T)$, and for two given Banach spaces~$X,Y$, the Bochner (sub)spaces
\begin{align}\notag 
L^p_{\rm loc}(\bbR_+;X)&\coloneqq\{f\mid\,  f\rest{I_0^T}\in L^p(I_0^T;X),\forall T>0 \},\notag\\
\bbW(I_0^T;X,Y)& \coloneqq\{f\in L^2(I_0^T;X)\mid\; \dot f\in L^2(I_0^T;Y)\},\notag\\
\bbW_{\rm loc}(\bbR_+;X,Y)& \coloneqq\{f\mid\; f\rest{I_0^T}\in \bbW(I_0^T;X,Y),\forall T>0\}.\notag
 \end{align}
\begin{lemma}\label{L:regSol}
If Assumptions~\ref{A:HV}-\ref{A:Arc} hold true and ~$f\in L^2_{\rm loc}(\bbR_+;V')$, then the solution of~\eqref{sys-yf} satisfies, for any given~$s\ge0$ and~$T>0$,
\begin{align}\notag
\norm{y}{\clC([s,s+T];H)}^2&\le C_{0,T}\norm{y}{\bbW(I_s^T;V,V')}^2\\
&\le C_{T}\left(\norm{y(s)}{H}^2+\norm{f}{L^2_{\rm loc}(I_s^T;V')}^2\right),\notag
\end{align}
with constants~$C_{0,T}=\ovlineC{C_{\rm rc}}\ge0$ and~$C_T=\ovlineC{C_{\rm rc}}\ge1$ independent of~$(s,y_0)$. 
\end{lemma}
The proof of Lemma~\ref {L:regSol} is standard, by following the Faedo--Galerkin method~\cite[Ch.~1, Sects.~1.3--1.4]{Temam01} (cf.~\cite[Sect.~3.4]{Rod21-jnls}).

Now, we consider a family of systems of type~\eqref{sys-yf},  with vanishing external forcing~$f$, indexed by the initial time~$s\ge0$ as follows,
\begin{equation}\label{sys-yvaric-s}\stepcounter{equation} \tag{\theequation.[s]}
 \dot z +Az+ A_{\rm rc}z =0,\quad z(s)= w(s),\quad t\ge s\ge 0,
 \end{equation}
where the initial state, at time~$s\ge0$, is given by the value of a function~$w(s)\in H$ at time~$s$. The following result will be essential for the proof of the main Theorem~\ref{T:main}, in particular, to derive appropriate regularity for the input.
\begin{lemma}\label{L:regSolvaric}
Let Assumptions~\ref{A:HV}--\ref{A:Arc} hold true, let~$\tau>0$ and~$w\in\clC([0,\infty),H)$, and let~$z=z_s(t)$ denote the solution of~\eqref{sys-yvaric-s}, $t\ge s$. Then, the function~$w^\tau(s)\coloneqq z_s(s+\tau)$ satisfies~$w^\tau\in\clC([0,\infty),H)$.
\end{lemma}
\begin{proof}
We start by defining $\underline{z}_s(t)\coloneqq z_s(t+s)$, for~$t\ge0$, and by noticing that, for~$(s,r)\in[0,\infty)$,
    \[
    w^\tau(s)-w^\tau(r)= z_s(s+\tau)-z_r(r+\tau)= \underline{z}_s(\tau)-\underline{z}_r(\tau).
    \]
    Next, for the difference~$\eta\coloneqq \underline{z}_s-\underline{z}_r$ we have the dynamics
    \[
    \dot\eta +A\eta+\underline A_{\rm rc,s}\eta=(\underline A_{\rm rc,s}-\underline A_{\rm rc,r})\underline{z}_r,\qquad t\ge0,
    \]
    where
    \[
    \underline A_{\rm rc,s}(t)\coloneqq A_{\rm rc}(t+s),\mbox{ for each } s\ge0\mbox{ and all } t>0.
    \]
    By Lemma~\ref{L:regSol} we find
    that
    \begin{align*}
     \norm{\eta(\tau)}{H}^2&\le C_{\tau}\left( \norm{\eta(0)}{H}^2+\norm{(\underline A_{\rm rc,s}-\underline A_{\rm rc,r})\underline{z}_r}{L^2((0,\tau),V')}^2\right)\\
     &\hspace{-.5em}\le C_{\tau}\left( \norm{\eta(0)}{H}^2+\dnorm{\underline A_{\rm rc,s}-\underline A_{\rm rc,r}}{}^2\norm{\underline{z}_r}{L^2((0,\tau),V)}^2\right)
      \end{align*}
      with~$\dnorm{g}{}\coloneqq\max_{t\in[0,\tau]}\norm{g(t)}{\clL(H,V')+\clL(V,H))}$. Using Lemma~\ref{L:regSol} again, we find
     \begin{align}
   &\varTheta\coloneqq\norm{w^\tau(s)-w^\tau(r)}{H}^2=\norm{\eta(\tau)}{H}^2\label{difwtausr}\\
     &\hspace{0em}\le C_{\tau}\left( \norm{\eta(0)}{H}^2+C_{0,\tau}\dnorm{\underline A_{\rm rc,s}-\underline A_{\rm rc,r}}{}^2\norm{z_r(r)}{H}^2\right)\notag\\
     &\hspace{0em}\le\! C_{\tau}\!\left( \norm{w(s)-w(r)}{H}^2+C_{0,\tau}\!\dnorm{\underline A_{\rm rc,s}-\underline A_{\rm rc,r}}{}^2\norm{w(r)}{H}^2\right)\!.\notag
    \end{align}

Let us fix arbitrary~$s\ge0$ and~$\epsilon>0$.
By the continuity of~$w$ it follows that there exists~$\delta_1>0$ such that
\[
\norm{w(r)-w(s)}{H}^2\le (2C_\tau\epsilon^2+1)^{-1},\quad\mbox{for}\quad r\in\clB_{\delta_1}^+(s),
\]
where~$\clB_{\gamma}^+(c)\coloneqq(c-\gamma,c+\gamma)\cap[0,\infty)$ denotes the ball of~$[0,\infty)$, with center~$c$ and radius~$\gamma$. Hence, using~\eqref{difwtausr},
\begin{align*}
 \varTheta   &\le 2^{-1}\epsilon^2+C_{\tau}C_{0,\tau}\dnorm{\underline A_{\rm rc,s}-\underline A_{\rm rc,r}}{}^2(\norm{w(s)}{H}+1)^2,
    \end{align*}
    for all $r\in\clB_{\delta_1}^+(s)$.
By the continuity of~$A_{\rm rc}$, in Assumption~\ref{A:Arc}, there exists~$\delta\in(0,\delta_1)$ such that
\begin{align*}
&\dnorm{\underline A_{\rm rc,s}-\underline A_{\rm rc,r}}{}^2\\
&=\sup_{t\in[0,\tau]}\norm{\underline A_{\rm rc}(s+t)-\underline A_{\rm rc}(r+t)}{\clL(H,V')+\clL(V,H)}^2\\&\le (2C_{\tau}C_{0,\tau}(\norm{w(s)}{H}+1)^2\epsilon^2)^{-1},
\end{align*}
for all $r\in\clB_{\delta}^+(s)$.
Therefore, we arrive at
\begin{align*}
     \norm{w^\tau(s)-w^\tau(r)}{H}^2&\le \epsilon^2,\qquad\mbox{for all}\quad r\in\clB_{\delta}^+(s),
\end{align*}
showing the continuity of~$w^\tau$ at~$t=s$.
This finishes the proof, since~$s\ge0$ is arbitrary.
\end{proof}

Next, we consider the nominal controlled dynamics under a  perturbation~$g$, for~$t\ge 0$,
\begin{equation}\label{sys-yFg} 
 \dot y +Ay+ A_{\rm rc}y =B K y + g,\qquad y(0)= y_0.
 \end{equation}
 \begin{lemma}\label{L:pert-stable}
Let Assumptions~\ref{A:HV}-\ref{A:Kstab} hold true and let~$g\in L^p(\bbR_+;H)$, $p\in(1,\infty]$. Then,  the solution of~\eqref{sys-yFg} satisfies, for~$t\ge s\ge 0$,
\begin{equation}\notag
\norm{y(t)}{H}\le D_1\rme^{-\mu_1(t-s)}\norm{y(s)}{H}+C_p\norm{g}{L^p((s,\infty);H)},
\end{equation}
with~$C_p=\ovlineC{C_{\rm rc}}\ge0$ independent of~$(t,s,y_0)$.
\end{lemma}
\begin{proof}
By the Duhamel formula, Assumption~\ref{A:Kstab}, and Definition~\ref{D:clAstab}, we have that
\begin{equation}\notag
\norm{y(t)}{H}\le D_1\rme^{-\mu_1(t-s)}\norm{z(s)}{H}+D_1{\textstyle\int_s^t}\rme^{-\mu_1(t-r)}\norm{g(r)}{H}\,\rmd r. 
\end{equation}
Further, by H\"older inequality, for the last term we find
\begin{align}\notag
&{\textstyle\int_s^t}\rme^{-\mu_1(t-r)}\norm{g(r)}{H}\,\rmd r\\
&\hspace{2em}\le({\textstyle\int_s^t}\rme^{-\frac{p}{p-1}\mu_1(t-r)}\,\rmd r)^\frac{p-1}{p}({\textstyle\int_s^t}\norm{g(r)}{H}^p\,\rmd r)^\frac1p\notag\\
&\hspace{2em}<\tfrac{p-1}{p\mu_1}\norm{g}{L^p((s,t);H)}. \notag
\end{align}
The result follows with~$C_p=D_1(p-1)(p\mu_1)^{-1}$.
\end{proof}

\section{Destabilizing effect of delayed feedback inputs}\label{S:dest-delayK}
We give an example for which Assumptions~\ref{A:HV}--\ref{A:Kstab} are satisfied and where there exists a threshold~$\tau_*>0$ so that no linear nominal stabilizing feedback~$K$ will remain stabilizing if subject to a time-delay~$\tau\ge\tau_*$.
This fact is not explicitly written in the literature. It is known that a fixed feedback input stabilizing  the nominal system can become destabilizing if subject to a large enough delay.

Let us set~$V=H=\bbR$, $A=\rho\Id$, $\Psi_1=1$, and~$B=\Id$, with~$\rho>0$. Then, the system
 \begin{align}\notag
 \dot y = \rho y,&
\qquad y(0)=y_{0},
\end{align}
is unstable and, in this case, exponentially stabilizing linear feedback input operators are necessarily of the form~$K=\kappa_0\Id$ with~$\kappa_0<-\rho$. Thus, let us fix~$\kappa<-\rho$. From~$BKy=\kappa y$, we arrive at the closed-loop nominal exponentially stable system
 \begin{align}\notag
 \dot y = \rho y+\kappa y,&
\qquad y(0)=y_{0}.
\end{align}
In case the same feedback input~$u(t)=\kappa y(t)$ is delayed by time~$\tau$, we will have the dynamics
 \begin{align}\label{sys-y-ex-Ktau} 
 \dot y(t) = \begin{cases}
\rho y(t),&\mbox{for }t\in[0,\tau),\\
\rho y(t) +\kappa y(t-\tau),&\mbox{for }t\in[\tau,\infty).
 \end{cases}
\end{align}
Next, we show that~\eqref{sys-y-ex-Ktau} is unstable if~$\tau\ge\rho^{-1}$. Let 
\begin{equation}\label{ode-intro-tauhat}
\widehat\tau\coloneqq\widehat\tau(\rho,\kappa)\coloneqq (\kappa^2-\rho^2)^{-\frac12}\arccos(-\tfrac\rho\kappa).
\end{equation}
From~$\rho>0$, we obtain $\kappa+\rho<0$ and~$\kappa-\rho<0$.
Thus we can apply the result in~\cite[Sect.4.5, Thm.~4.7(c)]{Smith11} to conclude that~\eqref{sys-y-ex-Ktau} is unstable for~$\tau>\widehat\tau$ and asymptotically stable for~$0\le\tau<\widehat\tau$. 
Next, since~$\rho>0$ and~$\kappa<-\rho$, we can write~$\kappa = -\gamma\rho$, with~$\gamma>1$ and, from
\begin{equation}\notag
 \widehat\tau(\rho,-\gamma\rho)= (\gamma^2-1)^{-\frac12}\rho^{-1}\arccos(\gamma^{-1})\ge 0,
\end {equation}
we obtain~$\lim_{\gamma\to\infty} \widehat\tau(\rho,-\gamma\rho)= 0$ and, by writing~$\gamma^{-1}\eqqcolon\cos(\theta)$, $\theta\in(0,\frac\pi2)$, we find
\begin{align}
\lim_{\gamma\searrow1} \widehat\tau(\rho,-\gamma\rho)&= \rho^{-1}\lim_{\theta\searrow0}\left((\cos(\theta)^{-2}-1)^{-\frac12}\theta\right)\notag\\
&= \rho^{-1}\lim_{\theta\searrow0}\left(\cos(\theta)\sin(\theta)^{-1}\theta\right)=\rho^{-1},\notag
\end{align}
and 
\begin{align}
\vartheta\coloneqq\rho^{-1}\tfrac{\rmd}{\rmd\gamma}\widehat\tau(\rho,-\gamma\rho)&=-\gamma(\gamma^2-1)^{-\frac32}\arccos(\gamma^{-1})\notag\\
&\quad+(\gamma^2-1)^{-\frac12}\gamma^{-2}(1-\gamma^{-2})^{-\frac12}\notag
\end{align}
which gives us
\begin{align}
\vartheta&=\gamma^{-1}(\gamma^2-1)^{-\frac32}\left(-\gamma^2\arccos(\gamma^{-1})+(\gamma^2-1)^{\frac12}\right)\notag\\
&=\gamma(\gamma^2-1)^{-\frac32}\left(-\arccos(\gamma^{-1})+\gamma^{-1}(1-\gamma^{-2})^{\frac12}\right..\notag
\end {align}
Since $\gamma^{-1}= \cos (\theta)$  it follows that 
\begin{align}
2\rho^{-1}\gamma^{-1}(\gamma^2-1)^{\frac32}\tfrac{\rmd}{\rmd\gamma}\widehat\tau(\rho,-\gamma\rho)&=-2\theta+2\cos(\theta)\sin(\theta)\notag\\
&\hspace{0em}=-2\theta+\sin(2\theta)<0.\notag
\end {align}
Therefore, we have
\begin{align}
&\lim_{\gamma\to\infty} \widehat\tau(\rho,-\gamma\rho)= 0,\quad\lim_{\gamma\searrow1} \widehat\tau(\rho,-\gamma\rho)= \rho^{-1}, \notag\\
\text{ and }&\quad\tfrac{\rmd}{\rmd\gamma}\widehat\tau(\rho,-\gamma\rho)<0,\quad\mbox{for }\gamma>1.\notag
\end {align}
In particular, we find~$\widehat\tau=\widehat\tau(\rho,\kappa)<\rho^{-1}$, independently of the chosen feedback gain~$\kappa<-\rho$.

\begin{remark}
The instability of~\eqref{sys-y-ex-Ktau}, for large~$\tau$, can be used to show the destabilizing effect of delayed feedback inputs for more general parabolic-like equations as well.
We give details in the  Appendix.
\end{remark}

\section{Stabilization}\label{S:stabil}

\subsection{Asymptotic null controllability}\label{sS:asynullct}
Let a feedback operator~$K$ as in Assumption~\ref{A:Kstab} be given  so that~$\clA_K\coloneqq -A-A_{\rm rc}+BK$ is~$(D_1,\mu_1)$-stable, for some~$D_1\ge1$ and~$\mu_1>0$. Then, since the delay~$\tau\ge0$ is known, we can easily find an input~$u(t)=u_{\rm olc}(t)$, such that its delayed version~$u_{\rm olc}^{[\tau]}(t)\coloneqq u_{\rm olc}(t-\tau)$ drives the state asymptotically to zero, exponentially fast. Namely, 
we solve the free-dynamics
 \begin{align}\label{ex-olp}
 &\dot p+Ap+A_{\rm rc} p =0,\quad p(0)= y_0,\quad\mbox{for } t\in(0,\tau),
 \end{align} 
 up to time~$\tau$ to find~$p(\tau)$, and set~$u_{\rm olc}(0)\coloneqq Kp(\tau)$. Then, we  solve the undelayed controlled system
\begin{align}\notag
 &\dot z+Az+A_{\rm rc}z =BKz,\quad z(\tau)= p(\tau),\quad\mbox{for } t\in(\tau,\infty)
 \end{align} 
and, finally, simply set
 \begin{align}\notag
 &u_{\rm olc}(t)\coloneqq Kz(t+\tau),&\mbox{for } t\in[0,\infty).
 \end{align} 

Then, the solution of
 \begin{align}\notag
 \begin{cases}
  \dot y+Ay+ A_{\rm rc}y =0,&\mbox{for }t<\tau,\\
    \dot y+Ay+ A_{\rm rc}y =Bu_{\rm olc}(t-\tau),&\mbox{for }t\ge\tau,
 \end{cases}
 \end{align} 
with initial state~$y(0)= y_0$, is given by
  \begin{align}\label{ex-oly}
&y(t)\coloneqq\begin{cases}
 p(t),&\mbox{for } t\in[0,\tau),\\
 z(t),&\mbox{for } t\in[\tau,\infty).
 \end{cases}
 \end{align} 
From~$\dot z=\clA_Kz$, we find~$\norm{y(t+\tau)}{H}^2=\norm{z(t+\tau)}{H}^2\le D_1\rme^{-\mu_1 t}\norm{y(\tau)}{H}^2\le D_1C_\tau\rme^{-\mu_1 t}\norm{y(0)}{H}^2$, with~$C_\tau$ given by Lemma~\ref{L:regSol}, used with~$(s,T,f)=(0,\tau,0)$. 
 
Though the delayed open-loop control~$u_{\rm olc}(t-\tau)$ above stabilizes the system, it is not robust against small errors. For example,  when solving numerically the free dynamics to compute~$p(\tau)$, we will,  obtain an approximation~$p(\tau)=y(\tau)-\varepsilon$ of the real state~$y(\tau)$. Hence, with the control input~$u=Kz$ above, the functions~$z$, $y$, and~$w\coloneqq y-z$, will satisfy, for~$t>\tau$,
\begin{align}\notag
 &\dot z+Az+ A_{\rm rc}z=BKz,\qquad z(\tau)=p(\tau),\notag
\\
 &\dot y+Ay+ A_{\rm rc}y =BKz,\qquad y(\tau)=y(\tau),\notag
\\
 &\dot w+Aw+ A_{\rm rc}w=0,\qquad\quad w(\tau)=\varepsilon,\notag
 \end{align}
The auxiliary state~$z(t)$ will converge exponentially to~$0$, but if the free dynamics is unstable,  then $w(t)$ will not necessarily converge asymptotically to~$0$. Thus, the real state~$y(t)=z(t)+w(t)$ will not necessarily converge asymptotically to~$0$.

Consequently,  the open-loop input above is not appropriate for practical applications,  since it does not take into account possible online state-measurement, estimation, and computation errors~$\varepsilon$, and consequently will not be able to respond to such errors, which can jeopardize its stabilizing properties.

\subsection{A predictor-based stabilizing feedback input}
Here, we shall construct an input that will be able to respond to state-measurement errors. In fact, we simply replace the open-loop control in Section~\ref{sS:asynullct} by a feedback control, constructed with the help of a predictor.
We commence by describing  the construction of  the feedback acting at $t=0$.  The state $y(\tau)$ can be computed by solving~\eqref{ex-olp} on $(0,\tau)$. Then,  the input is taken  depending  on the resulting~$y(\tau)$, thus, depending essentially only on~$y_0=y(0)$, leading to $u(0)= K_\tau(0,y(0))\coloneqq K y(\tau)$ as before.
Next, we proceed analogously for each time~$t>0$, resulting in a  control input  of the form ~$u(t)=K_\tau(t;y(t))$ depending on the state $ y(t)$.
Thus, we propose a state-feedback input~$u=u_\tau$ as follows,
\begin{subequations}\label{KfkY1}
\begin{align}
&u_\tau(s)\coloneqq 0,\quad\mbox{for}\quad s\in[-\tau,0);\label{KfkY1-Ks}\\
&u_\tau(t)\coloneqq  K_\tau(t;y(t)),\quad\mbox{for }t\ge0,\label{KfkY1-Kt}\\
\mbox{with}\quad& K_\tau(t,h)\coloneqq K\fkY(t,t+\tau, u_\tau^{[\tau]};h);
\intertext{where~$Y(t)\coloneqq\fkY(t_0,t,f;\fky)$ denotes the solution of}
&\dot Y+AY+A_{\rm rc}Y =Bf,\quad Y(t_0)=\fky,\label{KfkY-Y}
 \end{align}
 \end{subequations}
for~$t\ge t_0$.
Next, let us consider the system
 \begin{subequations}\label{sys-y-BKY1}
 \begin{align}
 &\dot y+Ay+ A_{\rm rc}y =0,\quad&& t<\tau,\\
 &\dot y+Ay+ A_{\rm rc}y =B K_\tau(t-\tau,y(t-\tau)),&& t>\tau,\\
  & y(0)= y_0,&&
 \end{align}
   \end{subequations}
with the time delayed input in~\eqref{KfkY1}, 
and the system
 \begin{subequations}\label{sys-w-BKY1}
 \begin{align}
 &\dot w+Aw+ A_{\rm rc}w =0,\quad&& t\in(0,\tau),\\
 &\dot w+Aw+ A_{\rm rc}w =B Kw,&& t>\tau,\\
  & w(0)= y_0,&&
 \end{align}
   \end{subequations}
where the nominal (undelayed) feedback input~$Kw$ is active only for time~$t\ge\tau$.

The next result concerns a property of the solutions of the  systems above, that seems to be well understood, though not explicitly written, in the literature.

\begin{lemma}
    The solutions of systems~\eqref{sys-y-BKY1} and~\eqref{sys-w-BKY1} coincide.
\end{lemma}
\begin{proof}
    We simply observe that, from~$u_\tau^{[\tau]}(t)=K_\tau(t-\tau,y(t-\tau))$, together with~\eqref{sys-y-BKY1} and the Duhamel formula, it follows that
    $y(t+\tau)=\fkY(t,t+\tau,u_\tau^{[\tau]},y(t))$, which means that in~\eqref{KfkY1-Kt} we have~$u_\tau(t)=K_\tau(t;y(t))=Ky(t+\tau)$. Thus, for~$t>\tau$, in~\eqref{sys-y-BKY1} we find
    \[
    B  K_\tau(t-\tau,y(t-\tau))=B u_\tau(t-\tau)=BKy(t),
    \]
    which agrees with the input in~\eqref{sys-w-BKY1}, for~$t>\tau$.
\end{proof}

Now, since the input~$u_\tau^{[\tau]}(t)=Ky(t)$ takes the form of a stabilizing feedback input for time~$t\ge\tau$, it will be able to respond to small state measurement errors.

\subsection{A predictor- and state-estimate-based stabilizing feedback input}

In case the initial state $y(0)$ is unknown, it has to be replaced by an estimate $\widehat y(0) $ and the input~$u(t)$ has to be constructed from a state-estimate~$\widehat y(t)$.

We propose the analogue of the input in~\eqref{KfkY}, based on a prediction of (an estimate of) the state at time~$t+\tau$ constructed from an estimate~$\widehat y(t)$ for~$y(t)$. That is, the estimate-based feedback input~$u=\widehat u_\tau$ is taken as follows,
\begin{subequations}\label{KfkY}
\begin{align}
&\widehat u_\tau(s)\coloneqq 0,\quad\mbox{for}\quad s\in[-\tau,0),\label{KfkY-Ks}\\
&\widehat u_\tau(t)\coloneqq \widehat K_\tau(t;\widehat y(t)),\quad\mbox{for}\quad t\ge0,\label{KfkY-Kt}\\
\mbox{with}\quad& \widehat K_\tau(t,h)\coloneqq K\fkY(t,t+\tau,\widehat u_\tau^{[\tau]};h).
\end{align}
 \end{subequations}
Thus ~$\fkY(t,t+\tau,\widehat u_\tau^{[\tau]};\widehat y(t))$ is a prediction of~$y(t+\tau)$ at time~$t+\tau$, based on the estimate~$\widehat y(t)$ of~$y(t)$, at time~$t$.

Next, we analyze the robustness of the corresponding system,  where~$\widehat y$ is provided by a Luenberger observer,  
 \begin{subequations}\label{sys-cl-BKY}
 \begin{align}
 \dot y+Ay+ A_{\rm rc}y &=0,&&\hspace{-.5em} t<\tau,\label{sys-cl-BKY-y1}\\
 \dot y+Ay+ A_{\rm rc}y &=B\widehat K_\tau(t-\tau,\widehat y(t-\tau)),&&\hspace{-.5em} t>\tau,\label{sys-cl-BKY-y2}\\
 \dot{\widehat y}+A\widehat y+ A_{\rm rc}\widehat y &=0,&&\hspace{-.5em} t<\tau,\label{sys-cl-BKY-haty1}\\
 \dot{\widehat y}+A\widehat y+ A_{\rm rc}\widehat y &=B\widehat K_\tau(t-\tau,\widehat y(t-\tau))&&\hspace{-.5em}\notag\\
 &\quad+L(W\widehat y-Wy),&&\hspace{-.5em} t>\tau,\label{sys-cl-BKY-haty2}\\
   y(0)= y_0,\quad\;\;&\quad\widehat y(0)= \widehat y_0,&&\hspace{-.5em}\label{sys-cl-BKY-ics}
 \end{align}
   \end{subequations}
with the input in~\eqref{KfkY}, delayed by time~$\tau$. We recall (cf.~\eqref{observer-intro}) that~$Wy$ represents the output of sensor measurements at our disposal, $L$ represents an output-injection operator, and~$\widehat y_0$ is an initial guess we might have for the unknown initial state~$y_0$.
It will  be convenient to express $u_\tau^{[\tau]}$ in detailed manner as 
\begin{align}\label{eq:kk1}
& \widehat u_\tau^{[\tau]}(t)= \widehat u_\tau(t-\tau)=K\fkY(t-\tau,t,\widehat u_\tau^{[\tau]}\rest{(t-\tau,t)};\widehat y(t-\tau))\notag\\
 &=K\fkY(t-\tau,t,\widehat u_\tau\rest{(t-2\tau,t-\tau)};\widehat y(t-\tau)),\text{ for } t >\tau,\hspace{-.5em}
 \end{align}
where $\widehat u_\tau(t)=0$  for $t<0$.
It will also be convenient to introduce the extension by zero as follows:
 \begin{align}
& \fkE_0\colon L^2(\bbR_+,X)\to L^2((-\tau,\infty),X),\notag\\
 &\fkE_0 f(t)\coloneqq\begin{cases}
 0,&\mbox{if }t\in(-\tau,0);\\ f(t),&\mbox{if }t>0.
 \end{cases}\notag
\end{align}

\begin{theorem}\label{T:main}
Let Assumptions~\ref{A:HV}--\ref{A:detect} hold true. Then, the solution of~\eqref{sys-cl-BKY} with the state-estimate based input~$\widehat u_\tau$ in~\eqref{KfkY} satisfies, for~$ t\ge s\ge0$,
 \begin{align}
  \norm{y(t)}{H} &\le C_1\rme^{-\mu_1 (t-s)}\norm{y(s)}{H}\notag\\
  &\quad+ C_2\norm{ \fkE_0(\widehat y-y)}{L^\infty((s-\tau, t-\tau);H)},\label{estT:main}
\intertext{where~$(D_1,\mu_1)$ is as in Assumption~\ref{A:Kstab}, and~$C_1\ge1$ and~$C_2\ge0$ are of the form}
&\hspace{-3em}C_1=\ovlineC{D_1,C_{\rm rc},\tau,\mu_1},\quad C_2=\ovlineC{C_{\rm rc},\tau,D_1\mu_1^{-1},\norm{BK}{\clL(H)}}.\notag
\end{align} 
Furthermore,~$\widehat u_\tau\in L^\infty(\bbR_+,\bbR^m) \cap C([0,\infty),\bbR^m)$.
\end{theorem}
\begin{proof}
We start by observing that the estimate error~$\eta\coloneqq \widehat y-y$ solves
\begin{subequations}\label{sys-esterr-BKY}
\begin{align}
  &\dot{\eta}+A\eta+ A_{\rm rc}\eta =LW\eta,\quad&& t>0;\\
 & \eta(0)= \eta_0\coloneqq \widehat y_0- y_0.&&
 \end{align}    
\end{subequations}
By Assumption~\ref{A:detect} it follows that
\begin{align}\label{est.eta}
\norm{\eta(t)}{H}\le C\norm{\eta(s)}{H},\quad\mbox{for all}\quad t\ge s\ge0.
\end{align}

From \eqref{eq:kk1} we have that
 \begin{align}
 \widehat u_\tau^{[\tau]}(t)&= K\fkY(t-\tau,t,\widehat u_\tau\rest{(t-2\tau,t-\tau)};\widehat y(t-\tau)),
 \label{hatu=KPred}
 \end{align}
for~$t>\tau$, where $\widehat u_\tau(t)=0$  for $t<0$. We write
 \begin{align}
&\hspace{-5em}\fkY(t-\tau,t,\widehat u_\tau\rest{(t-2\tau,t-\tau)};\widehat y(t-\tau))=\varkappa_1+\varkappa_2;\label{Pred-vkappas}\\
\mbox{with}\quad\varkappa_1(t)&=\fkY(t-\tau,t,\widehat u_\tau\rest{(t-2\tau,t-\tau)}; y(t-\tau));\notag\\
\mbox{and}\quad\varkappa_2(t)&=-\fkY(t-\tau,t,\widehat u_\tau\rest{(t-2\tau,t-\tau)}; y(t-\tau))\notag\\
&\quad+\fkY(t-\tau,t,\widehat u_\tau\rest{(t-2\tau,t-\tau)};\widehat y(t-\tau)).\notag
 \end{align}
 Note that~$\varkappa_1(t)=y(t)$ and, consequently,
 \begin{equation}\label{dyn-y-vkappa2}
 \begin{cases}\dot y+Ay+ A_{\rm rc}y =0,\quad&\mbox{ } t<\tau;\\
 \dot y+Ay+ A_{\rm rc}y =BKy+BK\varkappa_2,\quad&\mbox{  } t>\tau.
 \end{cases}
 \end{equation}
Next, it is important to observe that~$\varkappa_2(t_0)=z(t_0)$, where for every given~$t_0\ge\tau$, the function~$z$ follows the dynamics as in~\eqref{sys-yvaric-s},
\begin{equation}\label{dyn.vkappa2}
\dot z+Az+A_{\rm rc}z =0,\quad z(t_0-\tau)=\eta(t_0-\tau),
  \end{equation}
for~$t>t_0-\tau$. By Lemmas~\ref{L:regSol} and~\ref{L:regSolvaric}
we have that
 \begin{align}
 \norm{\varkappa_2(t_0)}{H}\le C_\tau\norm{\eta(t_0-\tau)}{H},\quad\text{for all }t_0\ge\tau,
 \label{cont.vkappa2}
  \end{align}
with~$C_\tau=\ovlineC{C_{\rm rc},\tau}\ge1$ and also that~$\varkappa_2\in\clC([\tau,\infty),H)$.

Denoting by~$\fkZ(s,t)h$ the solution of~$\dot w=(-A-A_{\rm rc}+BK)w$, with~$t\ge s$ and~$w(s)=h$, we use Duhamel formula to obtain
\begin{align}
 y(t)&=\fkZ(s,t)y(s)+\int_{s}^{t} \fkZ(r,t)BK\varkappa_2(r)\,\rmd r,\quad\mbox{ } t\ge s\ge\tau,\notag
  \end{align}
and from the stability of~$(-A-A_{\rm rc}+BK)$, given by Assumption~\ref{A:Kstab}, we obtain,
\begin{subequations}\label{squuez-st}
\begin{align}
\norm{y(t)}{H}&\le D_1\rme^{-\mu_1 (t-s)}\norm{y(s)}{H}\notag\\
&\quad+D_1\int_{s}^{t}\rme^{-\mu_1(t-r)} \norm{BK\varkappa_2(r)}{H}\,\rmd r\label{squuez-st-0}\\
&\hspace{-3em}\le D_1\rme^{-\mu_1(t-s)}\norm{y(s)}{H}\notag\\
&\hspace{-3em}\quad+{\norm{\widehat y-y}{L^\infty((s-\tau,t-\tau);H)}} D_1C_\tau\norm{BK}{\clL(H)}\mu_1^{-1},\notag\\
&\text{ for~$t\ge s\ge\tau$.}\label{squuez-st-1}
 \end{align}

By~\eqref{est.eta} we find $\norm{\widehat y-y}{L^\infty((s-\tau,t-\tau);H)}\le\norm{\eta}{L^\infty(\bbR_+;H)}\le C\norm{\eta_0}{H}<\infty$ and, by  Lemma~\ref{L:regSol},
 \begin{align}
 & \norm{y(t)}{H}\le  C_\tau\norm{y(s)}{H}\le  C_\tau\rme^{\mu_1\tau}\rme^{-\mu_1 (t-s)}\norm{y(s)}{H},\notag\\
 &\text{for~$0\le s\le \tau\le t$}, \label{squuez-st-2}
   \end{align} 
 with~$C_\tau\coloneqq\ovlineC{C_{\rm rc},\tau}\ge1$.
 {Finally, for~$0\le s\le \tau\le t$, by combining~\eqref{squuez-st-1} and~\eqref{squuez-st-2}, 
 \begin{align}
\norm{y(t)}{H}&\le D_1\rme^{-\mu_1(t-\tau)}\norm{y(\tau)}{H}\notag\\
&\quad+{\norm{\widehat y-y}{L^\infty((0,t-\tau);H)}} D_1C_\tau\norm{BK}{\clL(H)}\mu_1^{-1}\notag\\
&\le D_1C_\tau\rme^{\mu_1\tau}\rme^{-\mu_1(t-s)}\norm{y(s)}{H}\notag\\
&\quad+{\norm{\widehat y-y}{L^\infty((0,t-\tau);H)}} D_1C_\tau\norm{BK}{\clL(H)}\mu_1^{-1},\notag\\
&\mbox{for } 0\le s\le \tau\le t.\label{squuez-st-3}
\end{align}}
 \end{subequations}
 By the estimates in~\eqref{squuez-st} {and by~$D_1C_\tau\rme^{\mu_1\tau}\ge\max\{D_1,C_\tau\}\ge1$,} it follows that
     \begin{align}
  \norm{y(t)}{H}&\le C_1\rme^{{-\mu_1} (t-s)}\norm{y(s)}{H}\label{est-y}\\
&\quad+{\norm{\fkE_0(\widehat y-y)}{L^\infty((s-\tau,t-\tau);H)}} C_2,\quad  t\ge s\ge0,\notag
   \end{align} 
with~$C_1\coloneqq {D_1C_\tau\rme^{\mu_1\tau}}$ and~$C_2\coloneqq D_1C_\tau\norm{BK}{\clL(H)}\mu_1^{-1}$.

It remains to show the regularity of the input. From~\eqref{hatu=KPred} and~\eqref{Pred-vkappas} we have that $\widehat u_\tau^{[\tau]}(t)=K(\varkappa_1(t)+\varkappa_2(t))=K(y(t)+\varkappa_2(t))$ and by~\eqref{cont.vkappa2} we know that~$\varkappa_2\in\clC([\tau,\infty),H)$. 
Then, for an arbitrarily fixed~$T>\tau$,
by~\eqref{dyn-y-vkappa2} and standard regularity arguments for parabolic-like equations, we can conclude that~$y\in W((\tau,T),V,V')\subset\clC([\tau,\infty),H)$. 
Combining~\eqref{est-y}, \eqref{est.eta}, and~\eqref{cont.vkappa2}, we also find~$y+\varkappa_2\in L^\infty((\tau,\infty),H)$. Therefore, from~$K\in\clL(H,\bbR^m)$ we obtain~$\widehat u_\tau^{[\tau]}=K(y+\varkappa_2)\in\clC([\tau,\infty),\bbR^m)\cap L^\infty((\tau,\infty),\bbR^m)$. Consequently,~$\widehat u_\tau=\widehat u_\tau^{[\tau]}(\cdot+\tau)\in L^\infty(\bbR_+,\bbR^m)\cap\clC([0,\infty),\bbR^m)$.
  \end{proof}

   \begin{remark}
By choosing~$s=\frac t2$ in~\eqref{estT:main}, we arrive at
      \begin{align}
\limsup_{t\to\infty}\norm{y(t)}{H}&\!\le\! C_2\limsup_{t\to\infty}\norm{\fkE_0(\widehat y-y)}{L^\infty({(\frac{t}2-\tau,t-\tau)};H)}\notag\\
&\hspace{-3em}=C_2\lim_{t\to\infty}\norm{\widehat y-y}{L^\infty((t,\infty);H)}=C_2\eta_\infty,\label{asylim-yt}
   \end{align} 
with~$\eta_\infty\coloneqq{\lim\limits_{t\to\infty}\norm{\widehat y-y}{L^\infty((t,\infty);H)}}$, which implies that, given~$\gamma>C_2$, for large enough time~$t$, the norm of the state  remains in a neighborhood of zero with radius~$\rho=\gamma \eta_\infty$,  which is proportional  to the asymptotic upper bound~$\eta_\infty$  of the state-estimation error.
 \end{remark}

\section{Exponential observers} \label{S:observers}
Given an output operator~$W\in\clL(V,\bbR^s)$ we need to  design an output-injection operator~$L\in\clL(\bbR^s,V')$ such that Assumption~\ref{A:detect} is satisfied. We go a step further and seek an  
exponential observers, that is, we design~$L\in\clL(\bbR^s,V')$ so that~$-A-A_{\rm rc}+LW$ is a~$(D_2,\mu_2)$-stable operator, for some~$D_2\ge1$ and~$\mu_2>0$, see Definition~\ref{D:clAstab}. In this case we will have that the error~$\eta\coloneqq\widehat y-y$ decreases exponentially (cf.~\eqref{sys-esterr-BKY}. Further,~$\norm{y(t)}{H}$ will decrease exponentially to zero as well, as follows.

\begin{theorem}\label{T:exp}
 Let Assumptions~\ref{A:HV}--\ref{A:Kstab} hold true and let~$-A-A_{\rm rc}+LW$ be~$(D_2,\mu_2)$-stable. Then, for any given~$0<\mu\le\min\{\mu_2,\mu_1\}$ such that~$\mu<\max\{\mu_2,\mu_1\}$, there exist~$D\ge1$ such that the solution~$(y,\widehat y)$ of the closed-loop system~\eqref{sys-cl-BKY} satisfies, for~$t\ge s\ge0$,
 \begin{align}
&\norm{(y(t),\eta(t))}{H\times H}\le D\rme^{-\mu(t-s)}\norm{(y(s),\eta(s))}{H\times H},
\notag
\intertext{where~$\eta=\widehat y-y$ is the state estimate error. Furthermore, for a suitable constant~$D_0\ge0$, the control input satisfies}
&\norm{\widehat u_\tau(t)}{\bbR^m}\le D_0\rme^{-\mu(t-s)}\norm{(y(s),\eta(s))}{H\times H}.
\notag
\end{align}
\end{theorem}
\begin{proof}
With~$\eta\coloneqq \widehat y-y$, from~\eqref{sys-cl-BKY} we obtain, for~$t>0$,
\begin{subequations}\label{sys-cleta}
 \begin{align}
 &\dot y+Ay+ A_{\rm rc}y =0,\quad&& t<\tau,\label{sys-cleta-y1}\\
 &\dot y+Ay+ A_{\rm rc}y =B\widehat K_\tau(t-\tau,\widehat y(t-\tau)),&& t>\tau,\label{sys-cleta-y2}\\
 &\dot{\eta}+A\eta+ A_{\rm rc}\eta =LW\eta,&& t>0,\label{sys-cleta-eta}\\
  & \eta(0)= \widehat y_0-y_0.&&\label{sys-cleta-ic}
 \end{align}
   \end{subequations}
   For the estimate error~$\eta$ we have that
 \begin{equation}
   \norm{\eta(t)}{H}\le D_2\rme^{-\mu_2(t-s)}\norm{\eta(s)}{H},\qquad t\ge s\ge0.\label{squuez-exp-eta}  
 \end{equation}
Proceeding as in the proof of Theorem~\ref{T:main} we arrive at~\eqref{squuez-st-0}, for~$t\ge s\ge\tau$,
\begin{align}
 \norm{y(t)}{H}&\le D_1\rme^{-\mu_1 (t-s)}\norm{y(s)}{H}\notag\\
 &\quad+D_1\int_{s}^{t}\rme^{-\mu_1(t-r)} \norm{BK\varkappa_2(r)}{H}\,\rmd r\notag\\
 &\le D_1\rme^{-\mu_1 (t-s)}\norm{y(s)}{H}\notag\\
 &\quad+D_1\norm{BK}{\clL(H)}\int_{s}^{t}\rme^{-\mu_1(t-r)} \norm{\varkappa_2(r)}{H}\,\rmd r,
 \notag
 \end{align}
 with~$\varkappa_2$ satisfying~\eqref{cont.vkappa2}. Hence,
\begin{align}
 \norm{y(t)}{H}&\le D_1\rme^{-\mu_1 (t-s)}\norm{y(s)}{H}\notag\\
 &\quad+D_1\norm{BK}{\clL(H)}C_\tau\int_{s}^{t}\rme^{-\mu_1(t-r)} \norm{\eta(r-\tau)}{H}\,\rmd r\notag\\
 &= D_1\rme^{-\mu_1 (t-s)}\norm{y(s)}{H}\notag\\
 &\quad+C_0\norm{\eta(s)}{H}\int_{s}^{t}\rme^{-\mu_1(t-r)} \rme^{-\mu_2(r-s)}\,\rmd r.\notag
 \end{align}
 with~$C_0\coloneqq D_1\norm{BK}{\clL(H)}C_\tau D_2\rme^{\mu_2\tau}$.
 By~\cite[Prop.~3.2]{AzmiRod20}, we we have that for any~$\mu$ satisfying~$0<\mu\le\min\{\mu_2,\mu_1\}$ and~$\mu<\max\{\mu_2,\mu_1\}$, there exists~$C_1>0$ such that
 \begin{align}
 \norm{y(t)}{H}&\le D_1\rme^{-\mu_1 (t-s)}\norm{y(s)}{H}\notag\\
 &\quad+C_0C_1\rme^{-\mu (t-s)}\norm{\eta(s)}{H}.\label{squuez-exp-y}
 \end{align}
 By combining~\eqref{squuez-exp-y} and~\eqref{squuez-exp-eta}, we find
  \begin{align}
 &\norm{(y(t),\eta(t))}{H}^2=\norm{y(t)}{H}^2+\norm{\eta(t)}{H}^2\notag\\
 &\le 2D_1^2\rme^{-2\mu_1 (t-s)}\norm{y(s)}{H}^2+2C_0^2C_1^2\rme^{-2\mu (t-s)}\norm{\eta(s)}{H}^2\notag\\
 &\quad+D_2^2\rme^{-2\mu_2(t-s)}\norm{\eta(s)}{H}^2\notag\\
 &=D\rme^{-2\mu (t-s)}\norm{(y(s),\eta(s))}{H\times H}^2,\notag
 \end{align}
 with~$D\coloneqq \max\{2D_1^2,2C_0^2C_1^2+D_2^2\}$.

Finally, for the input, we find
\[
\norm{\widehat u_\tau^{[\tau]}(t)}{\bbR^m}\le \norm{K}{\clL(H,\bbR^m)}\norm{\fkY(t-\tau,t,\widehat u_\tau^{[\tau]},\widehat y(t-\tau))}{H}
\]
and we write again
\begin{align*}
 \Xi(t)&\coloneqq \fkY(t-\tau,t,\widehat u_\tau^{[\tau]},\widehat y(t-\tau))\notag\\
 &= \fkY(t-\tau,t,\widehat u_\tau^{[\tau]},y(t-\tau))+\fkY(t-\tau,t,0,\eta(t-\tau))\\
  &= y(t)+\varkappa_2(t),
\end{align*}
with~$\varkappa_2$ satisfying~\eqref{cont.vkappa2}. Hence, by~\eqref{squuez-exp-y}, 
\begin{align}
\norm{\Xi(t)}{H}&\le C_3\rme^{-\mu(t-s)}\norm{(y(s),\eta(s))}{H\times H}+C_\tau^\frac12\norm{ \eta (t-\tau))}{H}\notag
\end{align}
with~$C_3\coloneqq \max\{D_1,C_0C_1\}$. Now, \eqref{squuez-exp-eta} gives us, for~$t\ge s+\tau$,
\begin{align*}
\norm{ \eta (t-\tau))}{H}
&\le D_2\rme^{\mu\tau}\rme^{-\mu(t-s)}\norm{\eta(s)}{H}
\end{align*}
and we arrive at
\begin{align}
\norm{\Xi(t)}{H}&\le C_4\rme^{-\mu(t-s)}\norm{(y(s),\eta(s))}{H\times H},\mbox{ for } t\ge s+\tau,\notag
\end{align}
 with~$C_4\coloneqq C_3+C_\tau^\frac12D_2\rme^{\mu\tau}$. Therefore, for~$t\ge s\ge0$,
 \begin{align*}
 \norm{\widehat u_\tau(t)}{\bbR^m}&=\norm{\widehat u_\tau^{[\tau]}(t+\tau)}{\bbR^m}\notag\\
 &\le C_4\norm{K}{\clL(H,\bbR^m)}\rme^{-\mu(t+\tau-s)}\norm{(y(s),\eta(s))}{H\times H},
 \end{align*}
which finishes the proof.
  \end{proof}

Below, we shall consider sensors performing average-like measurements and shall give suitable output-injection operators  explicitly. 

 \begin{remark}\label{R:observer-errors}
 In practice the output~$w(t)$ will be subject to small sensor measurement errors. This means that likely we will not have a vanishing asymptotic limit for the state estimate error as~${\widehat e_\infty\coloneqq\lim\limits_{t\to\infty}\norm{\widehat y-y}{L^\infty((t,\infty);H)}}=0$ in~\eqref{asylim-yt}. Instead, we expect  that~$\widehat e_\infty=\widehat e_\infty(\zeta_{\rm mag})$ will be a constant depending on the magnitude~$\zeta_{\rm mag}$ of the error of those sensor measurements, hopefully so that~$\widehat e_\infty(\zeta_{\rm mag})\to0$ as~$\zeta_{\rm mag}\to 0$. Furthermore, in applications, the model~\eqref{sys-cl-BKY-haty1}--\eqref{sys-cl-BKY-haty2} will likely be a numerical approximation of~\eqref{sys-cl-BKY-y1}--\eqref{sys-cl-BKY-y2}, thus~$\widehat e_\infty$ may also depend on the numerical discretization errors. 
   \end{remark}

\section{Satisfiability of the assumptions} \label{S:satAssum}
We consider a  concrete parabolic equation as
\begin{align*}
 &\tfrac{\partial}{\partial t} y -(\nu\Delta-\Id) y+ay +b\cdot\nabla y
 =B u^{[\tau]},\quad
  \fkB  y\rest{\Gamma}=0,
 \end{align*}
evolving in the pivot space is~$H\coloneqq L^2(\Omega)$, with initial state~$y(0)=y_0\in H$.
By taking~$A\coloneqq-\nu\Delta+\Id$ as the shifted-scaled Laplacian under the given boundary conditions defined by~$\fkB$ (e.g., of Neumann or of Dirichlet type)  and by taking~$A_{\rm rc}\coloneqq a\Id +b\cdot\nabla$, we see that Assumptions~\ref{A:HV}--\ref{A:Arc} are satisfied.

Now, we show that the satisfiability of Assumptions~\ref{A:detect} and~\ref{A:Kstab} follows from the results in~\cite{KunRodWal21,Rod21-aut}, for  large enough numbers of appropriately placed actuators and sensors. See Fig.~\ref{fig:act} (cf.~\cite[Fig.~1]{Rod21-aut}) as an illustration for a spatial rectangular domain, where a rescaled copy of the configuration corresponding to~$M=1$ is taken in partitions of the rectangular domain (an analogue domain-partition based strategy can be applied to more general convex polygonal domains as well; see~\cite[Rem.~2.8]{AzmiKunRod23-ieee}). 
  \begin{figure}[htbp]%
    \centering%
             {\includegraphics[width=1\textwidth]{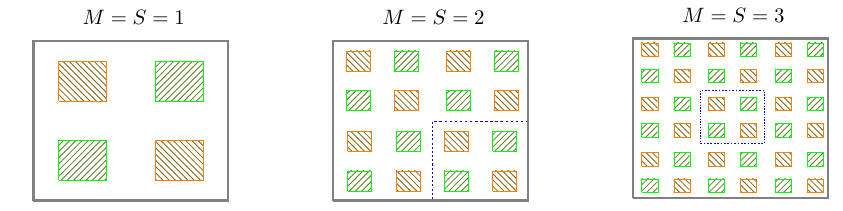}} 
        \caption{Actuators (``\slash''-pattern) and sensors (``\textbackslash''-pattern)}%
     \label{fig:act}%
\end{figure}

Let~$U_M=\{\indf_{\omega_{j}^{1,M}}\mid 1\le j\le m\}$ denote the set of actuators, given by indicator functions with supports~$\omega_{j}^{1,M}$ as in Fig.~\ref{fig:act}, for a given~$M\in\bbN_+$, and let~$\clU_M\coloneqq\linspan U_M$. 
Analogously, we denote the set of sensors by~$W_S=\{\indf_{\omega_{j}^{2,S}}\mid 1\le j\le S_\varsigma\}$, $S\in\bbN_+$, and also~$\clW_S\coloneqq\linspan W_S$. Note that, in Fig.~\ref{fig:act} we have~$m=2M^2$ actuators and~$s=2S^2$ sensors.
The control operator~$B=U_M^\diamond\colon\bbR^m\to H$ is
\begin{equation}\label{B=UM}
U_M^\diamond u(t)\coloneqq {\textstyle\sum\limits_{j=1}^{m}}u_j(t)\indf_{\omega_{j}^{1,M}}\in\clU_M,
\end{equation}
and the output operator $W =W_S^\vee\colon H\to \bbR^s$ is
\begin{equation}\label{W=WS}
W_S^\vee y(t)\coloneqq \left({\textstyle\int_{\omega_1^{2,S}}}y(t)\,\rmd x,\;\dots,\;{\textstyle\int_{\omega_{s}^{2,S}}}y(t)\,\rmd x\right),
\end{equation}
giving us the output of average-like sensors measurements.
 Further, let~$P_{\clF}\in\clL(H)$ denote the orthogonal projection in~$H$ onto a closed subspace~$\clF\subseteq H$. 
As a corollary of~\cite[Cor.~3.2 (with~$\clK_M(t,p)\coloneqq p$)]{KunRodWal21}, we have the following.
\begin{lemma}\label{L:KMstab}
Let Assumptions~\ref{A:HV}--\ref{A:Arc} hold true. Then, for each ~$\mu_1>0$, there exists a sufficiently~$M$ and~$\lambda>0$, such that the
operator~$\clA_{U_M}=-A-A_{\rm rc}-\lambda P_{\clU_M}$ is~$(D_1,\mu_1)$-stable, for some constant~$D_1\ge1$.
\end{lemma}
\begin{lemma}\label{L:LSstab}
Let Assumptions~\ref{A:HV}--\ref{A:Arc} hold true. Then, for each~$\mu_2>0$, there exists a sufficiently large~$S$ and~$\lambda>0$, such that the
operator~$\clA_{W_S}=-A-A_{\rm rc}-\lambda P_{\clW_S}$ is~$(D_2,\mu_2)$-stable, for some constant~$D_2\ge1$.
\end{lemma}

\begin{remark}
The result in~\cite[Cor.~3.2]{KunRodWal21} is based on an additional assumption on an auxiliary set of functions~$\widetilde U_M$ as in~\cite[Assum.~2.5]{KunRodWal21}. These auxiliary functions can be taken as bump-like ``regularized indicator functions'' as in~\cite[Sect.~6, Eq.~(6.3)]{Rod21-aut}.
\end{remark}

Concerning  Assumption~\ref{A:detect}, we observe that it is satisfied if~$\clA_{LW}=-A- A_{\rm rc}+LW$ is $(D_2,\mu_2)$-stable. 
Hence, to show the satisfiability of Assumptions~\ref{A:detect} and~\ref{A:Kstab}, it is sufficient to observe that the projections in Lemmas~\ref{L:KMstab} and~\ref{L:LSstab} can be written as
\begin{equation}\label{seekKL}
-\lambda P_{\clU_M}=U_M^\diamond K\quad\mbox{and}\quad -\lambda P_{\clW_M}=LW_S^\vee,
\end{equation}
for an appropriate feedback input operator~$K\in\clL(V,\bbR^{m})$ and an appropriate output injection operator~$L\in\clL(\bbR^{s}, V')$. To find such a pair~$(K,L)$ it is convenient to introduce~$W_S^\diamond\colon\bbR^s\to\clW_S$ and~$U_M^\vee\colon H\to\bbR^m$ as
\begin{align}\label{UM-inv}
W_S^\diamond v&\coloneqq {\textstyle\sum\limits_{i=1}^{s}}v_iv\indf_{\omega_{i}^{2,S}};
\\
\label{WS-inv}
U_M^\vee h&\coloneqq \left({\textstyle\int_{\omega_1^{1,M}}}h\,\rmd x,\;\dots,\;{\textstyle\int_{\omega_{m}^{1,M}}}h\,\rmd x\right).
\end{align}

\begin{lemma}\label{L:BU-PU}
We have the identities
\begin{equation}\label{BU-PU}
P_{\clU_M}=U_M^\diamond  \clV_M^{-1} U_M^\vee,\qquad P_{\clW_S}=W_S^\diamond  \clV_S^{-1} W_S^\vee,
\end{equation}
where~$\clV_{1,M}=[\clV_{1,M,(i,j)}]\in\bbR^{m\times m}$ and~$\clV_{2,S}=[\clV_{2,S,(i,j)}]\in\bbR^{s\times s}$ are the symmetric matrices with entries~$\clV_{1,M,(i,j)}\coloneqq (\indf_{\omega_{j}^{1,M}},\indf_{\omega_{i}^{1,M}})_H$ and~$\clV_{2,S,(i,j)}\coloneqq (\indf_{\omega_{j}^{2,S}},\indf_{\omega_{i}^{2,S}})_H$ in the $i$-th row and~$j$-th column.
\end{lemma}
The statement of Lemma~\ref{L:BU-PU} is a corollary of~\cite[Lem.~2.8]{KunRod19-cocv}. 
The  sought pair~$(K,L)$ satisfying~\eqref{seekKL} is
\begin{align*}\notag
K&=K_M\coloneqq-\lambda \clV_M^{-1} U_M^\vee\in\clL(H,\bbR^{M_\sigma}),\\
L&=L_S\coloneqq-\lambda W_S^\diamond  \clV_S^{-1}\in\clL(\bbR^{S_\varsigma},H).
\end{align*}

\section{Numerical results} \label{S:numer}

We take diffusion-reaction-convection parameters as
\begin{align}
&\nu=0.1,\qquad b(x,t)=\begin{bmatrix}x_1+x_2\\|\cos(6t)x_1x_2|_\bbR\end{bmatrix},\notag \\
 &a(x,t)=-\tfrac32+x_1 -|\sin(6t+x_1)|_\bbR, \notag
\end{align}
and choose Neumann boundary conditions, $\fkB=\bfn\cdot\nabla$, in the spatial domain $\Omega=(0,1)\times(0,1)\subset\bbR^2$.

With~$A\coloneqq-\nu\Delta+\Id$ and~$A_{\rm rc}\coloneqq a\Id +b\cdot\nabla$ the simulations correspond to system~\eqref{sys-cl-BKY},
 \begin{subequations}\label{sys-y-BKY-zeta}
 \begin{align}
 &\dot y+Ay+ A_{\rm rc}y =0,\qquad\qquad\qquad t<\tau,\label{sys-y-BKY-zeta-d1}\\
 &\dot y+Ay+ A_{\rm rc}y =B\widehat u_\tau^{[\tau]},\qquad\qquad t>\tau,\label{sys-y-BKY-zeta-d2}\\
  & y(0)= y_0,&&\label{sys-y-BKY-zeta-ic}\\
&\mbox{with } \widehat u_\tau^{[\tau]}(t)=K\fkY(t-\tau,t,\widehat u_\tau^{[\tau]};\widehat y(t-\tau)).&&\label{sys-y-BKY-zeta-fkY}
   \end{align}
   \end{subequations}

 For the temporal discretization of this system two time steps, $t^{\rm s}$ for the predictor solver $\fkY$ in~\eqref{sys-y-BKY-zeta-fkY}, and $ t^{\rm s}_{\rm r}$ for  system~\eqref{sys-y-BKY-zeta-d1}--\eqref{sys-y-BKY-zeta-ic} with $t^{\rm s}_{\rm r} <  t^{\rm s}\le \frac{1}{2} \tau$ are introduced.  The reason for solving~\eqref{sys-y-BKY-zeta} in a temporal mesh finer than that of the predictor~$\fkY$ is to highlight the fact that in practice the predictor-based input will indeed be obtained numerically, while system~\eqref{sys-y-BKY-zeta} will be running in real (continuous) time.

The information of the predictor-solver~$\fkY$ will be updated at the time instances~$n t^{\rm s}$, $n\in\bbN$. It is then used to compute the control input which is  held constant over the interval $[nt^{\rm s},(n+1) t^{\rm s}) $ for the system \eqref{sys-y-BKY-zeta-d1}--\eqref{sys-y-BKY-zeta-ic}.

The  state-estimate is provided by an observer
 \begin{subequations}\label{sys-haty-BKY-num}
 \begin{align}
 &\dot {\widehat y}+A\widehat y+ A_{\rm rc}\widehat y =L(W\widehat y-w_\zeta),&& t<\tau;\\
 &\dot {\widehat y}+A\widehat y+ A_{\rm rc}\widehat y =B\widehat u_\tau^{[\tau]}+L(W\widehat y-w_\zeta),&& t>\tau;\\
  & \widehat y(0)= \widehat y_0;\qquad w_\zeta=Wy+\eta;&&
 \end{align}
   \end{subequations}
 as in~\eqref{sys-cl-BKY-haty1}--\eqref{sys-cl-BKY-haty2}, where we include an additional perturbation~$\zeta$ in order to model sensor measurement errors (we think of~$w(t)=Wy(t)$ as the correct output and of~$w(t)+\zeta(t)$ as the measured one). We consider the same time-step as the predictor and simulate the dynamics by taking an unbiased noise of the form 
$\zeta(n t^{\rm s})=\zeta_{\rm mag}(-1+2{\tt rand}_n)\in\bbR^{S_\varsigma}. $
 Here~$\zeta_{\rm mag}\ge0$ is a constant that determines the magnitude of the output error. The random vector~$v={\tt rand}_n$ at time~$n t^{\rm s}$ is  generated by the Matlab function {\tt rand}, thus~$z\coloneqq\zeta(n t^{\rm s})$ has entries~$z_i\in[-\zeta_{\rm mag},\zeta_{\rm mag}]$, $1\le i\le N$.  
Concerning reproducibility, as random number generators we have used the Matlab function~{\tt rng}, with seed~$1$, and the Mersenne Twister generator.
Finally, as initial states we have taken
\[
y_0(x_1,x_2)\coloneqq 1-2x_1x_2\quad\mbox{and}\quad \widehat y_0(x_1,x_2)\coloneqq -1-3x_2^2.
\]

\subsection{Actuators and sensors}
We consider the case of~$8$ actuators and~$8$ sensors.  Their supports together with the triangulation~$\fkT$ used to compute the  estimate provided by observer~\eqref{sys-haty-BKY-num} and the predictor-based input~\eqref{sys-y-BKY-zeta-fkY} are shown in Fig.~\ref{fig:mesh_act_pred}. 
 \begin{figure}[htbp]%
    \centering%
        {\includegraphics[width=.4\textwidth]{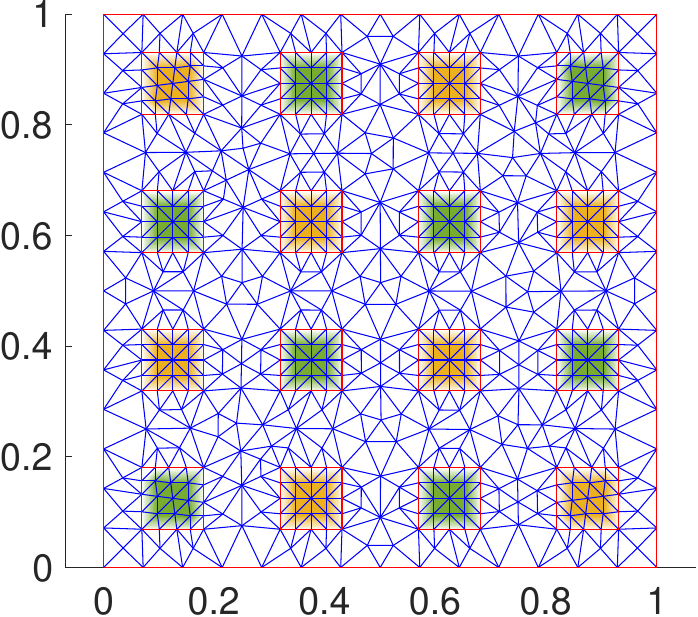}} 
  \caption{Spatial triangulation~$\fkT$.}%
     \label{fig:mesh_act_pred}%
\end{figure}

To better approximate the continuous-time dynamics of the infinite-dimensionality physical system~\eqref{sys-y-BKY-zeta-d1}--\eqref{sys-y-BKY-zeta-ic}, we will simulate it  on regular refinements of the mesh in Fig.~\ref{fig:mesh_act_pred}.
The nominal output-injection~$L$ and feedback-input~$K$ operators are taken as in~\eqref{seekKL},
\begin{equation}\notag
-\lambda_{ L} P_{\clW_M}\eqqcolon LW_S^\vee\quad\mbox{and}\quad -\lambda_{ K} P_{\clU_M}\eqqcolon U_M^\diamond K,
\end{equation}
with~$\Lambda=(\lambda_{ L},\lambda_{ K})=(200,100)$, where~$W_S$ denotes the set of sensors and~$U_M$ denotes the set of actuators. The free dynamics~$\Lambda=(0,0)$ will be considered as well.

\subsection{Spatio-temporal meshes}
As coarsest spatial approximation we use piecewise-linear finite elements (hat functions) corresponding to the triangulation~$\fkT$ in Fig.~\ref{fig:mesh_act_pred}. As coarsest temporal mesh we use a regular partition
\[
(nt^\rms)_{n=0}^\infty,\quad\mbox{with the time step}\quad t^\rms=10^{-3}.
\]
Thus, the coarsest spatio-temporal mesh is defined by the pair~$\fkM\coloneqq(\fkT,t^\rms)$.  
\begin{itemize}
\item The predictor-based input~\eqref{sys-y-BKY-zeta-fkY} and the observer~\eqref{sys-haty-BKY-num} are solved in the mesh~$\fkM$.
\item The real system~\eqref{sys-y-BKY-zeta} is simulated in refined  meshes
\[
\fkM_{\rm rf}=(\fkT_{\rm rf}, 2^{-(2+{\rm rf})}t^{\rms}),\qquad {\rm rf}\in\{0,1,2,3,4\},
\]
where~$\fkT_{0}\coloneqq\fkT$ and~$\fkT_{i+1}$ denotes a regular refinement of~$\fkT_{i}$.
\end{itemize}

\subsection{Results of simulations}
We  demonstrate the stabilizing performance of the proposed predictor--observer based delayed feedback input.
Firstly, in Fig.~\ref{fig:free-nom} we show that the free-dynamics is exponentially unstable.
 \begin{figure}[htbp]%
    \centering%
        {\includegraphics[width=.45\textwidth]{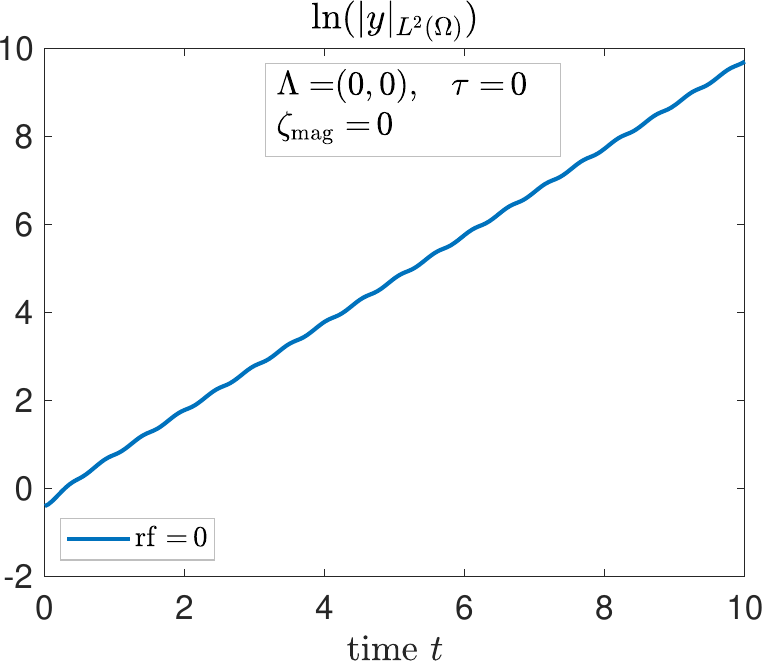}}
         \qquad%
         {\includegraphics[width=.45\textwidth]{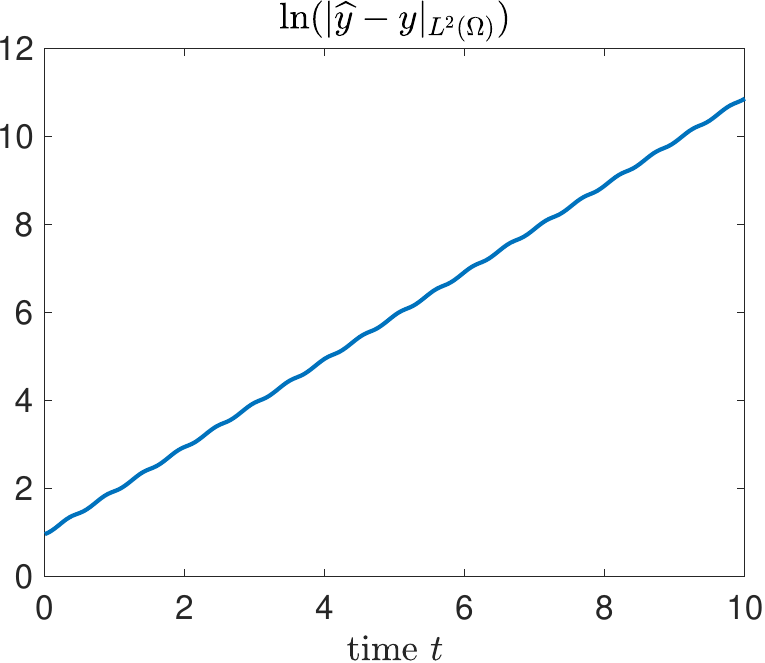}}
  \caption{Instability of the free dynamics and free observer.}%
     \label{fig:free-nom}%
\end{figure}

Then, in Fig.~\ref{fig:feed-delay} we can see that the the nominal (i.e., with~$\tau=0$)  closed-loop system with injection-feedback gain pair~$\Lambda=(200,100)$ is stable. In the same figure we show that the same feedback input becomes destabilizing as the delay increases. In Fig.~\ref{fig:feed-delay-zoom} we zoom the evolution of the norms for small instants for a better visualization of the response of the activation of the delayed input, with different values of~$\tau$.
\begin{figure}[htbp]%
    \centering%
         \subfigure[Norm of the state.]
         {\includegraphics[width=.45\textwidth]{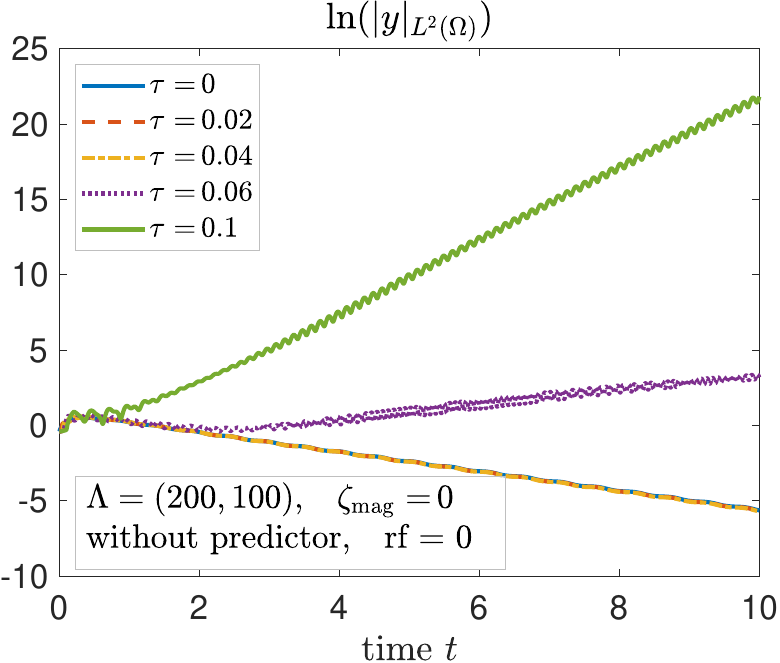}} %
         \qquad\subfigure[Norm of the state-estimate error.]
         {\includegraphics[width=.45\textwidth]{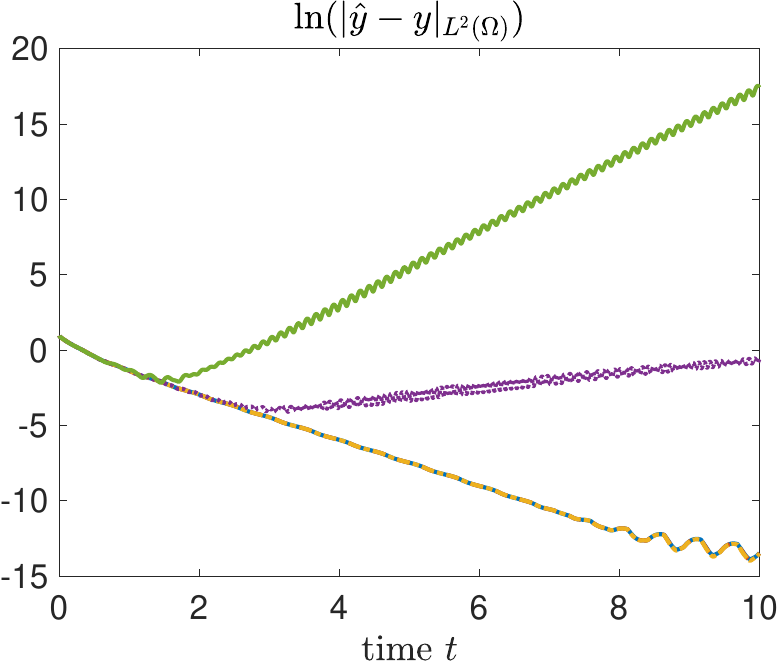}}
          \quad
        \caption{Evolution for feedback inputs delayed by time-$\tau$.}%
     \label{fig:feed-delay}%
\end{figure}
 \begin{figure}[htbp]%
    \centering%
         {\includegraphics[width=.8\textwidth]{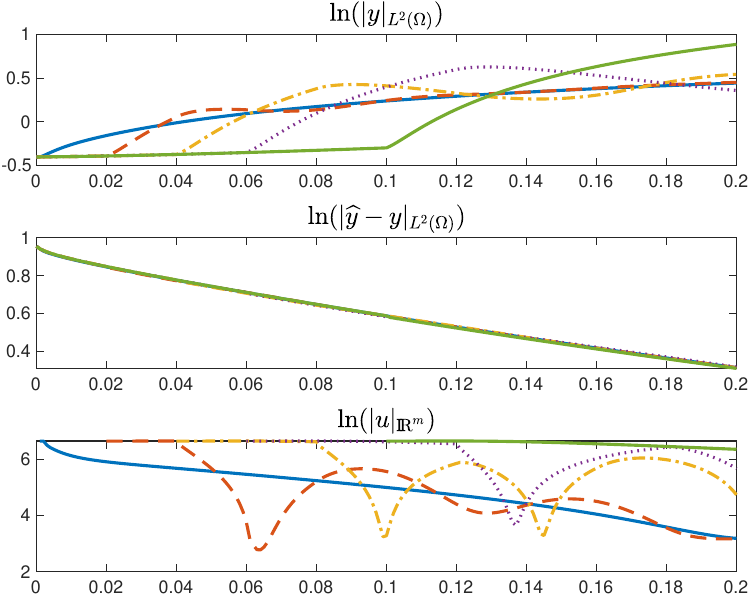}}
         \qquad
\caption{Time-zoom of the norms evolution in Fig~\ref{fig:feed-delay}.}%
     \label{fig:feed-delay-zoom}%
\end{figure}

To counteract the destabilizing effects of the time-delay observed in~Fig.\ref{fig:feed-delay}, we will include the time-$\tau$ predictor within the computation of the input.  The stabilizing contribution of the predictor is shown in Fig.~\ref{fig:feed-delay-pred}, 
\begin{figure}[htbp]%
    \centering%
         \subfigure[Norm of the state.]
         {\includegraphics[width=.45\textwidth]{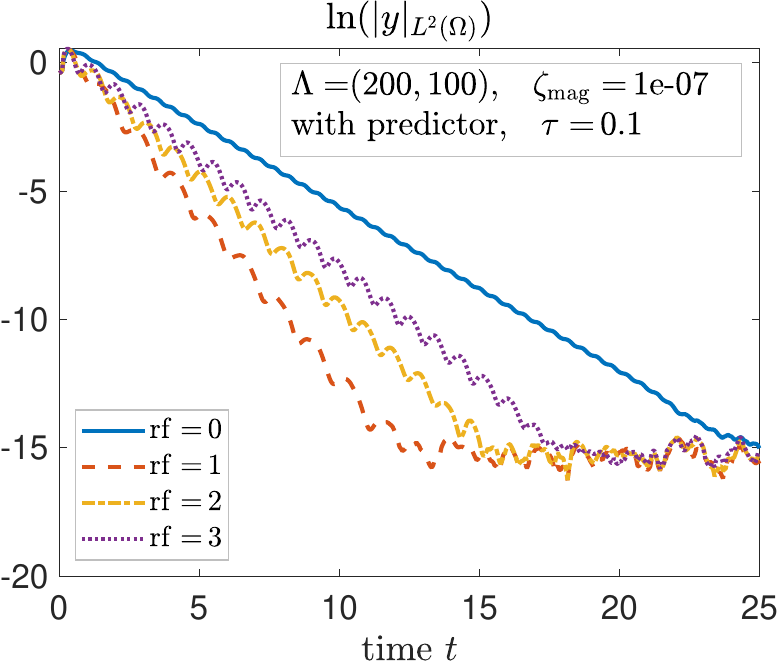}}%
         \qquad   \subfigure[Norm of the state-estimate error.\label{fig:feed-delay-pred-est}]
         {\includegraphics[width=.45\textwidth]{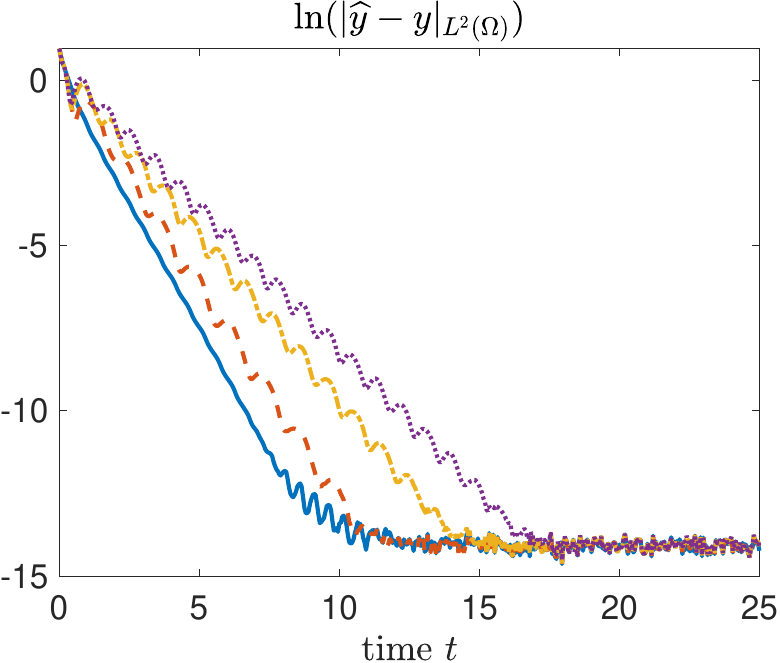}}
\caption{With delayed feedback inputs, using a predictor.}%
     \label{fig:feed-delay-pred}%
\end{figure}
for~$\tau=0.1$, where we see that the instability observed in Fig.~\ref{fig:feed-delay} is, indeed, counteracted by the use of the time-$\tau$-ahead state predictions.  

Further, we simulate the ``real'' system~\eqref{sys-y-BKY-zeta} in the spatio-temporal refinements~$\fkM_{\rm rf}$, $0\le{\rm rf}\le 3$ of the mesh~$\fkM_{0}$. Still, the predictor~\eqref{sys-y-BKY-zeta-fkY} and the observer~\eqref{sys-haty-BKY-num} are solved in the coarse mesh~$\fkM$. The results depicted in Fig.~\ref{fig:feed-delay-pred} show that the input obtained by using the estimate provided by the observer,  is able to stabilize~\eqref{sys-y-BKY-zeta} with the help of the predictor.

In Fig.~\ref{fig:feed-delay-pred-T1} 
\begin{figure}[htbp]%
    \centering%
          {\includegraphics[width=.8\textwidth]{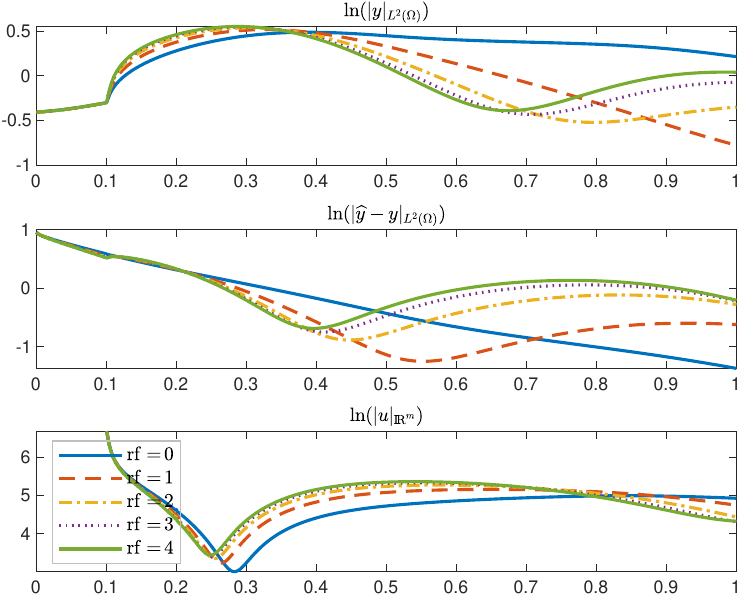}}
         \qquad
     \caption{Time-zoom including the norms in Fig.~\ref{fig:feed-delay-pred}.}%
     \label{fig:feed-delay-pred-T1}%
\end{figure}
we zoom on the behavior for small time instants, again to better see the response to the activation of the input. Further, we have run an extra simulation for one more refinement of the mesh used to simulate the real system. In particular, we can observe  convergence of the evolution of the norms as the mesh is refined.

Note that the magnitude~$\zeta_{\rm mag}$ of the state-dependent test-noise~$\zeta$ satisfies~$\ln(\lambda_1\zeta_{\rm mag})=\ln(200\cdot 10^{-7})\approx-10.8198$, which is the magnitude of the state-dependent noisy component of the output. This explains the oscillations observed for large time, and also the fact that the norms plotted in Fig.~\ref{fig:feed-delay-pred} do not drop below a certain threshold.

We also have to take into consideration the  measurement errors associated with the numerical output operators. Roughly speaking, in the experiments, the ``accuracy'' of the average measurements depend on the spatial mesh. Thus the (discretized) output operator in the coarse mesh has a different accuracy than that of the analogue (discretized) output operator in a refined mesh. Consequently there is a state-dependent measurement error as well. The results in Fig.~\ref{fig:feed-delay-pred} shows  the robustness of the constructed input against such state-dependent errors as well. 

Note also that even in the case~${\rm rf}=0$ we have the presence of a state-dependent numerical error, leading to a state-dependent output error. Indeed,  though the mesh~$\fkM_{0}$ uses the same triangulation as~$\fkM$, it still differs in the time-step, namely~$\frac14t^\rms$ versus~$t^\rms$. This could explain the oscillations starting at time~$t=8$ in Fig.~\ref{fig:feed-delay-pred-est} for the case~${\rm rf}=0$.

Finally, in Fig.~\ref{fig:feed-delay-noise} we show the response of the constructed input against the magnitude of the state-independent noise~$\zeta$. As expected the norms~$\norm{y(t)}{L^2(\Omega)}$ of the state and~$\norm{\widehat y(t)-y(t)}{L^2(\Omega)}$ of the state-estimate error converge to a neighborhood of zero with size proportional to the magnitude~$\zeta_{\rm mag}$ of the noise. 
\begin{figure}[htbp]%
    \centering%
         \subfigure[Norm of the state.]
         {\includegraphics[width=.45\textwidth]{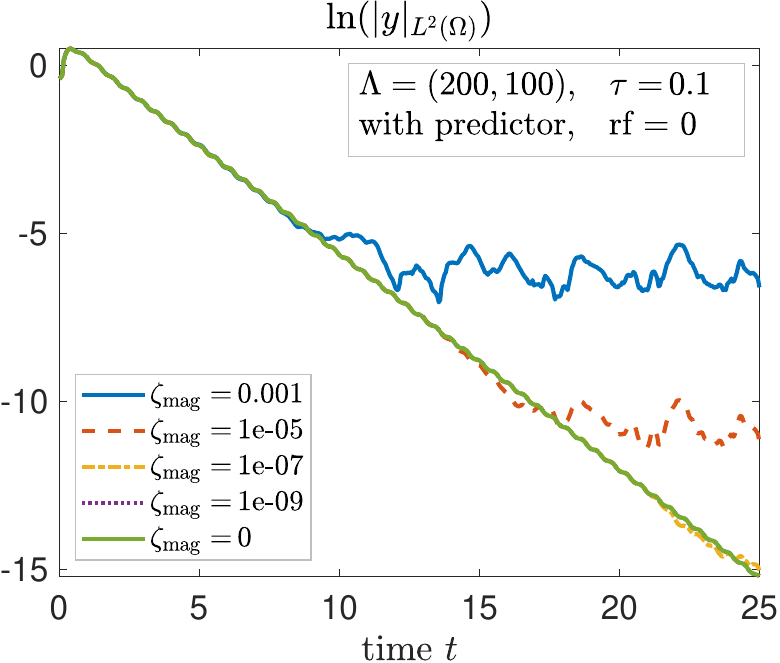}}%
         \qquad        \subfigure[Norm of the state-estimate error.]
         {\includegraphics[width=.45\textwidth]{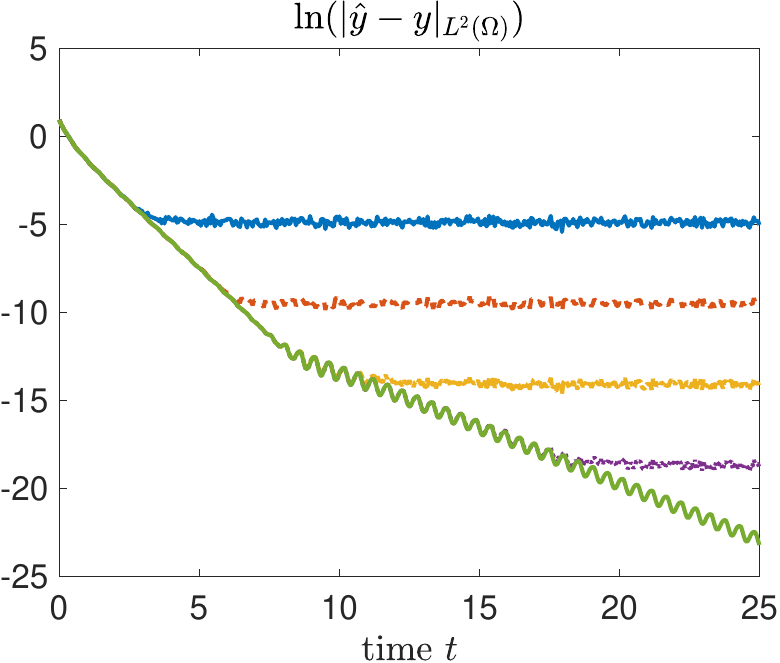}} 
  \caption{Effect of the magnitude~$\zeta$ of the noise in the output.}%
     \label{fig:feed-delay-noise}%
\end{figure}

\section{Concluding remarks}\label{S:finremks}

Assuming that the time delay~$\tau\ge0$ of the input is known, we have proposed a strategy to compute a control input, stabilizing linear parabolic equations, when subject to the time delay~$\tau$.  The construction of the input  utilizes a prediction of the controlled state at a future time~$t+\tau$, based on estimates of the controlled state at the present time~$t$. Asymptotically, as~$t\to\infty$, the controlled state converges exponentially to a neighborhood of zero proportional to the magnitude of the state estimation error. In particular, for exponential observers, the control input is exponentially stabilizing in the ideal situation in which the state-estimate error vanishes.

To compute the input~$K_\tau(t,\widehat y(t))$ as in~\eqref{KfkY-Kt} we use the prediction
$
Y(t+\tau)=\fkY(t,t+\tau,\widehat u_\tau^{[\tau]};\widehat y(t)),
$
of the state at time~$t+\tau$, for each time~$t>0$. In practice, after temporal discretization, and assuming a regular partition of the time-interval, this means computing $Y(n t^{\rms}+\tau)$, for each~$t=nt^{\rms}$ where~$t^{\rms}$ denotes the time-step of the (discretized) observer and predictor. 
Also, we can use several predictor-solvers running simultaneously, each corresponding to a time interval~$[nt^\rms,nt^\rms+\tau]$.

\appendix
\section*{Appendix}
\setcounter{section}{1}
\setcounter{theorem}{0} \setcounter{equation}{0}
\numberwithin{equation}{section}

The scalar {\sc ode} example in Section~\ref{S:dest-delayK} can be used  to show the destabilizing effect of time-delayed inputs, for a more general class of autonomous parabolic-like equations, including heat equations. Indeed, let~$(H,V,A)$  satisfy Assumptions~\ref{A:HV} and~\ref{A:A}. For example, $H=L^2(\Omega)$, $V=W^{1,2}(\Omega)$, and $A$ is the shifted scaled Neumann Laplacian~$-\nu\Delta+\Id:V\to V'$ as in~\eqref{sys-y-BKY-zeta}.

By Assumptions~\ref{A:HV} and~\ref{A:A}, it follows that~$A$ has a compact inverse~$A^{-1}\colon H\to H$. So, we can fix a sequence~$(\alpha_i,e_i)_{i\in\bbN_+}$ of eigenpairs of~$A$, 
\begin{align}
&Ae_i=\alpha_{i}e_i;\qquad0<\alpha_{i}\le\alpha_{i+1};\qquad\lim_{i\to\infty}\alpha_{i}=\infty;\notag\\
&(e_i)_{i\in\bbN_+}\quad\mbox{forms an orthonormal basis in } H.\notag
\end{align}

Let us fix~$\rho>0$ and choose~$m$ large enough, so that~$\alpha_{m+1}>\rho+\alpha_1$. Then, as actuators we take the
first~$m$ eigenfunctions, $E_m=\{e_i\mid 1\le i\le m\}.$

We define~$\underline K_m\in\clL(H,\bbR^{m})$, as
\begin{equation}\notag
h\mapsto \underline K_m h\coloneqq (v_1,v_2,\dots,v_{m}) ,\qquad v_i\coloneqq (e_i,h)_H,
\end{equation}
and  take~$A_{\rm rc}$ (satisfying Assumption~\ref{A:Arc}) and~$K$ as
 \begin{align}\notag
A_{\rm rc}=-(\rho+\alpha_1)\Id, \qquad  K=\kappa  \underline K_m,\quad\kappa<-\rho<0.
\end{align} 
In this case, system~\eqref{sys-y-intro0}, with vanishing delay,  becomes
 \begin{align}\notag
 \dot y(t)  =-(A-(\rho+\alpha_1)\Id-\kappa B \underline K_m)y(t) ,\qquad y(0)= y_0,
\end{align} 
with~$Bz\coloneqq{\textstyle \sum\limits_{i=1}^m}z_ie_i$ as in~\eqref{B-intro}.
We can see that~$B \underline K_m=P_{\clE_m}$. Hence,
we can write
 \begin{equation}\label{sys-y-appx-ex-1} 
 \begin{split}
  &\dot y(t)  =\Xi_\kappa y(t) ,\\
  &\mbox{with }\Xi_\kappa\coloneqq-(A-(\rho+\alpha_1)\Id-\kappa P_{\clE_m}).
 \end{split}
\end{equation} 
Note that the eigenfunctions of~$\Xi_\kappa$ coincide with those of~$A$, $\Xi_\kappa e_i=\beta_{\kappa,i} e_i$ with  eigenvalues 
\begin{align}
\beta_{\kappa,i}&=-\alpha_i+(\rho+\alpha_1)+\kappa,\quad&& 1\le i\le m;\notag\\
\beta_{\kappa,i}&=-\alpha_i+(\rho+\alpha_1),\quad&&  i\ge m+1.\notag
\end{align} 
\begin{enumerate}[label=(\roman*)]
\item[] \hspace{-2em} We observe the following:
\item The free dynamics (i.e., with~$\kappa=0$)  is unstable, because~$\beta_{0,1}=\rho>0$. 
\item The controlled nominal dynamics (i.e., with $\tau=0$), is exponentially stable, because
\begin{align}
\beta_{\kappa,i}&\le \rho+\kappa<0,\quad&& 1\le i\le m,\notag\\
\beta_{\kappa,i}&\le-\alpha_{m+1}+(\rho+\alpha_1)<0,\quad&&  i\ge m+1,\notag
\end{align} 
thus, for$\overline\beta\coloneqq\max\{\rho+\kappa,-\alpha_{m+1}+(\rho+\alpha_1)$, we have~$\beta_i\le\overline\beta<0$ for all~$i\in\bbN_+$.
\item\label{item-instdelay} If~$\tau>\rho^{-1}$, the controlled delayed dynamics is exponentially unstable, 
for reasons given below.
\end{enumerate}
With~$u^{[0]}(t)=Ky(t)=-\kappa P_{\clE_m}y(t)$ subject to time-delay~$\tau$, ~\eqref{sys-y-appx-ex-1} becomes, with~$A_1\coloneqq-(A-(\rho+\alpha_1)\Id)$,
 \begin{align}\label{sys-y-appx-ex-2} 
   \dot y(t)  =\begin{cases}
 A_1y(t),&\mbox{  }t<\tau,\\
A_1y(t) +\kappa P_{\clE_m} y(t-\tau),&\mbox{ }t>\tau,
 \end{cases}
 \end{align} 
Observe that~$y_1(t)\coloneqq (y(t),e_1)_H$  is the scalar corresponding to the coordinate~$y_1(t)e_1\coloneqq P_{\clE_1}y(t)$ of the orthogonal projection, in~$H$, of the solution  onto the space~$\clE_1=\bbR e_1$ spanned by the first eigenfunction of~$\Xi$.  We find that~$y_1(t)$ satisfies the dynamics in~\eqref{sys-y-ex-Ktau}, that is,
 \begin{align}\label{sys-y-appx-ex-3} 
 \dot y_1(t) &= \begin{cases}
\rho y_1(t),&\mbox{for }t\in[0,\tau),\\
\rho y_1(t) +\kappa y_1(t-\tau),&\mbox{for }t\in[\tau,\infty),
 \end{cases}\\
 y_1(0)&= (y(0),e_1)_H.
\end{align}
Further, if we take~$y_0=y_1(0)e_1$, then~$y(t)=y_1(t)e_1$. From the instability of~\eqref{sys-y-appx-ex-3} (i.e., of~\eqref{sys-y-ex-Ktau}) shown in Section~\ref{S:dest-delayK}, when~$\tau>\rho^{-1}$,  the instability of~\eqref{sys-y-appx-ex-2} follows.

\renewcommand*{\bibfont}{\normalfont\small}
\bibliographystyle{plainurl}
\bibliography{Delay}

\end{document}

%% file: Mathcommands.tex
\newcommand{\linspan}{\mathop{\rm span}\nolimits}

\newcommand{\rest}{\left.\kern-2\nulldelimiterspace\right|_}
\newcommand{\norm}[2]{\left|#1\right|_{#2}}
\newcommand{\dnorm}[2]{\left\|#1\right\|_{#2}}

\newcommand{\Id}{{\mathbf1}}
\newcommand{\indf}{1}

\newcommand{\clA}{{\mathcal A}}
\newcommand{\clB}{{\mathcal B}}
\newcommand{\clC}{{\mathcal C}}

\newcommand{\clE}{{\mathcal E}}
\newcommand{\clF}{{\mathcal F}}

\newcommand{\clK}{{\mathcal K}}
\newcommand{\clL}{{\mathcal L}}

\newcommand{\clU}{{\mathcal U}}
\newcommand{\clV}{{\mathcal V}}
\newcommand{\clW}{{\mathcal W}}

\newcommand{\bbN}{{\mathbb N}}

\newcommand{\bbR}{{\mathbb R}}

\newcommand{\bbW}{{\mathbb W}}

\newcommand{\fkB}{{\mathfrak B}}

\newcommand{\fkE}{{\mathfrak E}}

\newcommand{\fkM}{{\mathfrak M}}

\newcommand{\fkT}{{\mathfrak T}}

\newcommand{\fkY}{{\mathfrak Y}}
\newcommand{\fkZ}{{\mathfrak Z}}

\newcommand{\rmD}{{\mathrm D}}

\newcommand{\bfn}{{\mathbf n}}

\newcommand{\rmd}{{\mathrm d}}
\newcommand{\rme}{{\mathrm e}}

\newcommand{\rms}{{\mathrm s}}

\newcommand{\fkp}{{\mathfrak p}}

\newcommand{\fky}{{\mathfrak y}}

\newcommand{\ovlineC}[1]{\overline C_{\left[#1\right]}}

\definecolor{DarkBlue}{rgb}{0,0.08,0.45}
\definecolor{DarkRed}{rgb}{.65,0,0}
\definecolor{applegreen}{rgb}{0.55, 0.71, 0.0}

\newcounter{mymac@matlab}
\newcommand{\matlab}{MATLAB%
   \ifnum\value{mymac@matlab}<1%
   \textregistered%
   \setcounter{mymac@matlab}{1}%
   \fi%
  }
